\documentclass[11pt,a4paper]{article}

\usepackage{mathtools}
\usepackage{authblk} 
\usepackage{algpseudocode,algorithmicx,algorithm}

\usepackage{mathrsfs}
\usepackage{latexsym,bm}
\usepackage{amsmath,amsfonts,amsmath,amssymb,amsthm}
\usepackage{extarrows}

\usepackage{graphicx,subfigure,epstopdf,float}
\usepackage{enumerate,cases,multirow}
\usepackage{makecell}
\usepackage{caption}
\usepackage{tikz}
\usepackage{pgfplots}
\usepackage{appendix}
\usepackage{qcircuit}
\usepackage{longtable,colortbl,arydshln,threeparttable}
\definecolor{mygray}{gray}{.9}
\usepackage[shortlabels]{enumitem}

\usepackage{indentfirst}
\usepackage[top=25mm,bottom=20mm,left=25mm,right=20mm]{geometry}
\usepackage{cite}

\usepackage{listings}

\usepackage{makeidx}        
\usepackage{booktabs}
\usepackage[bookmarksnumbered,colorlinks,citecolor=red,linkcolor=red,hyperindex,linktocpage=true]{hyperref}

\newcommand{\Be}{{\boldsymbol{e}}}
\newcommand{\Bf}{{\boldsymbol{f}}}
\newcommand{\Bg}{{\boldsymbol{g}}}

\newcommand{\Bu}{{\boldsymbol{u}}}

\newcommand{\Bw}{{\boldsymbol{w}}}

\newcommand{\Ce}{\mathcal{E}}

\newcommand{\Cg}{\mathcal{G}}

\newcommand{\Cj}{\mathcal{J}}

\newcommand{\Cm}{\mathcal{M}}

\newcommand{\Cp}{\mathcal{P}}

\newcommand{\Cu}{\mathcal{U}}

\newcommand{\Cw}{\mathcal{W}}

\newcommand{\bbC}{\mathbb{C}}

\newcommand{\bbR}{\mathbb{R}}

\newcommand{\bbN}{\mathbb{N}}

\newcommand{\bbZ}{\mathbb{Z}}

\newcommand{\D}{\mathrm{d}}

\newcommand{\ket}[1]{| #1 \rangle} 

\newcommand{\vect}{\boldsymbol}

\def \d {\mathrm{d}}

\newcounter{parentalgorithm}

\makeatother

\newtheorem{theorem}{Theorem}[section]
\newtheorem{lemma}{Lemma}[section]

\newtheorem{definition}{Definition}[section]

\theoremstyle{remark}
\newtheorem{remark}{\bf Remark}[section]
\theoremstyle{assumption}

\numberwithin{equation}{section}

\begin{document}

	\title{Schr\"odingerisation based computationally stable algorithms for  ill-posed  problems in partial differential equations} 

\author[1,2,3]{Shi Jin   \thanks{shijin-m@sjtu.edu.cn}}	
\author[2,3,4]{Nana Liu
\thanks{nanaliu@sjtu.edu.cn}}
\author[1]{Chuwen Ma  
\thanks{chuwenii@sjtu.edu.cn}
\footnote{Corresponding author.}}	 
	\affil[1]{School of Mathematical Sciences,   Shanghai Jiao Tong University, Shanghai 200240, China.}
	\affil[2]{Institute of Natural Sciences, Shanghai Jiao Tong University, Shanghai 200240, China.}
	\affil[3]{Ministry of Education, Key Laboratory in Scientific and Engineering Computing, Shanghai Jiao Tong University, Shanghai 200240, China.} 
\affil[4] {University of Michigan-Shanghai Jiao Tong University Joint Institute, Shanghai 200240, China}


\maketitle
\begin{abstract}
	We introduce a simple and stable computational method for ill-posed partial differential equation (PDE) problems. The method is based on Schr\"odingerisation, introduced in [S. Jin, N. Liu and Y. Yu, arXiv:2212.13969][S. Jin, N. Liu and Y. Yu, Phys. Rev. A, 108 (2023), 032603], which maps all linear PDEs into Schr\"odinger-type equations in one higher dimension, for quantum simulations of these PDEs.  Although the original problem is ill-posed, the Schr\"odingerised equations are Hamiltonian systems and time-reversible, allowing stable computation both forward and backward in time. The original variable can be recovered by data from suitably chosen domain in the extended dimension.  We will use the (constant and variable coefficient) backward heat equation and the linear convection equation with imaginary wave speed as examples. Error analysis of these algorithms are conducted and verified numerically. The methods are applicable to both classical and quantum computers, and we also lay out quantum algorithms for these methods.  Moreover, we introduce a smooth initialisation for the Schr\"odingerised equation which will lead to essentially spectral accuracy for the approximation in the extended space, if a spectral method is used.  Consequently, the extra qubits needed due to the extra dimension, if a qubit based quantum algorithm is used, for both well-posed and ill-posed problems, becomes almost
	$\log\log {1/\varepsilon}$ where $\varepsilon$ is the desired precision. This optimizes the complexity of the Schr\"odingerisation based quantum algorithms for any non-unitary dynamical system introduced in [S. Jin, N. Liu and Y. Yu, arXiv:2212.13969][S. Jin, N. Liu and Y. Yu, Phys. Rev. A, 108 (2023), 032603].
\end{abstract}
\textbf{MSCcodes:}	65J20,  65M70, 81-08

\section{Introduction}

In this paper, we are interested in numerically computing ill-posed problems that follow the evolution of a general dynamical system
\begin{equation}\label{eq:ODE1}
	\frac{\D}{\D t} \Bu = H \Bu 
\end{equation}
 whose input data  given by $\Bu_0\in \bbR^n$.
For simplicity, we first consider  $H$ being Hermitian with  $n$ real eigenvalues, ordered as	\begin{equation}\label{eq:eigenvalues H1}
	\lambda_1(H)\leq \lambda_2(H)\leq \cdots\lambda_{n}(H), \quad 
	\text{for}\; \text{all}\; t\in [0,T].
\end{equation}
We assume that  $\lambda_i(H)>0$, for some $1\le i\le n$, which implies that the system \eqref{eq:ODE1} contains unstable modes, thus the initial value problem  is ill-posed  since its solution will grow exponentially in time.   This causes significant computational challenges. Normally the numerical errors will grow exponentially unless special care is taken.

Ill-posed or unstable problems appear in many physical applications, for example, fluid dynamics instabilities such as Rayleigh-Taylor and Kevin-Helmholtz instabilities \cite{kull1991theory, funada2001viscous}, plasma instability \cite{hasegawa2012plasma}, and Maxwell equations with negative index of refraction \cite{shelby2001experimental, parazzoli2003experimental}. It also appears in inverse problems \cite{uhlmann2014inverse, bertero2021introduction}. 
One of the most classical ill-posed problems is the backward heat equation which suffers from the catastrophic Hadamard instability. Usually, some regularization technique is used to make the problems well-posed \cite{engl1989convergence}. 

In this paper,  we study a generic yet simple computational strategy to numerically compute ill-posed problems based on Schr\"odingerisation, introduced for quantum simulation for dynamical systems whose evolution operator is non-unitary \cite{JLY22a, JLY22b}.
The idea is to map it to one higher dimension, into Schr\"odinger-type equations, that obeys unitary dynamics and thereby naturally fitting for quantum simulation.  Since the Schr\"odinger-type equations are Hamiltonian systems that are time-reversible, they can be solved  both forwards and backwards  in time in a computationally stable way. This makes it suitable for solving unstable problems, as was proposed in our previous work \cite{JLM24}.
In this article, we study this method for the classical backward heat equation, and a linear convection equation with imaginary wave speed (or negative index of refraction). 

The initial-value problem to the  backward heat equation  is  ill-posed in all three ways: (i) the solution does not necessarily exist; (ii) if the solution exists, it is not necessarily unique; (iii) there is no continuous dependence of the solution on arbitrary input data \cite{Na58,John55,Hollig83, lin2003remarks}.  This problem is well-posed for final data whose Fourier spectrum has compact support \cite{Miranker61}.
However, even when the solution exists and is unique, computing the solution is difficult since unstable physical systems usually lead to unstable numerical methods.  There have been various treatments of this difficult problem, for example quasi-reversibility methods \cite{XFQ06, Show74, LL76, AGE98, Miller73}, regularization methods \cite{Hao94,Mera01,Show74, biccari2023two}, Fourier truncation methods \cite{QFS07}, etc.  Our approach computes solution consistent with the Fourier truncation method. 

The actual implementation of our backward heat equation solver is as follows. First we start with the Fourier transform of the input data denoted by $u_T$ and truncate the Fourier mode in finite domain. This is either achieved automatically if the Fourier mode of $u_T$ has compact support, or we choose a sufficiently large domain in the Fourier space such that outside it the Fourier coefficient of  $u_T$ is sufficiently small. The latter can usually be done since the forward heat equation gives a solution that is smooth for $t>0$ so its Fourier coefficient decays rapidly. The Schr\"odingerisation technique lifts the backward heat equation to a Schr\"odinger equation in one higher dimension that is {\it time reversible}, which can be solved backward in time by any reasonable stable numerical approximation. We then recover $u(t)$ for $0\le t<T$ by integrating or pointwise evaluation over suitably chosen domain in the extended variable. Since the time evolution is based on solving the Schr\"odinger equation, the computational method is stable. 

We point out that the truncation in the Fourier space regularizes the original ill-posed problem. Although the initial-value problem to the Schr\"odingerised equation is well-posed even without this Fourier truncation, to recover $u$ one needs {\it finite} Fourier mode. We also show that the strategy also applies to variable coefficient heat equation.

We will also apply the  same strategy  to solve the unstable linear convection equation that has imaginary wave speed. 

We will give error estimates on these methods and conduct numerical methods that verify the results of the error analysis.  The methods work for both classical and quantum computers. Since the original Schr\"odingerisation method was introduced for quantum simulation of general PDEs, we will also give the quantum implementation of the computational method.

Moreover, we introduce a smooth initialization for the Schr\"odingerised equation, so the initial data will be in $C^k(\bbR)$ in the extended space for any integer 
$k\ge 1$ (see section \ref{higher-p}).  This will lead to essentially the spectral accuracy for the approximation in the extended space, if a spectral method is used.  Consequently, the extra qubits needed due to the extra dimension, if a qubit based quantum algorithm is used,
for both well-posed and ill-posed problems, becomes almost
$\log\log {1/\varepsilon}$ where $\varepsilon$ is the desired precision. This {\it reduces} the complexity of Schr\"odingerisation based quantum algorithms introduced in \cite{JLY22a, JLY22b} for any non-unitary  dynamical system.

The rest of the paper is organized as follows. In section~2, we give a brief review of the Schr\"odingerisation approach for general linear ODEs. In section 3, we study the backward heat equation, for both constant and variable coefficient cases,  and show the approximate solution based on the framework of Schr\"odingerisation. In section 4, the numerical method and error analysis of the backward heat equation are presented. In addition, a $C^k$ smooth initial data with respect to the extended variable is constructed to improve the convergence rates in the extended space. In section 5, we apply this technique to the convection equations with purely imaginary wave speed. In section~6, we show the numerical tests to verify our theories. Section 7 shows the quantum algorithms  and the corresponding complexity.

Throughout the paper, we restrict the simulation to a finite time interval $t\in [0,T]$.
The notation $f\lesssim g$ stands for $f\leq Cg$ where $C$ is independent of the mesh size and time step.
Scalar-valued quantities and vector-valued quantities are denoted by normal symbols and boldface symbols, respectively.
Moreover, we use a 0-based indexing, i.e. $j= \{0,1,\cdots,N-1\}$, or $j\in [N]$, and  $\Be_{j}^{(N)}\in \bbR^N$, denotes a vector with the $j$-th component being $1$ and others $0$.
We shall denote the identity matrix and null matrix by $I$ and $\textbf{0}$, respectively,
and the dimensions of these matrices should be clear from the context, otherwise,
the notation $I_N$ stands for the $N$-dimensional identity matrix.

\section{The general framework}\label{sec:The general framework}

We start with the basic framework set up in \cite{JLM24}   for forward problems which we first briefly review here. 

Using the warped phase transformation $\Bw(t,p) = e^{-p}\Bu$ for $p>0$ and symmetrically extending the initial data to $p<0$, one has the following system of linear convection equations \cite{JLY22a, JLY22b}:
\begin{equation}\label{eq:up}
	\frac{\D}{\D t} \Bw = -H \partial_p \Bw  \qquad 
	\Bw(0,p)= e^{-|p|}\Bu_0.
\end{equation}
This is clearly a hyperbolic system. When the eigenvalues of $H$ are all negative, the convection term of \eqref{eq:up} corresponds to waves moving from the right to the left.  One does, however, need a boundary condition on the right-hand side. 
Since $\Bw$ decays exponentially in $p$, one just needs to select $p=p^R$ large enough,  so $w$ at this point is essentially zero  and a zero incoming boundary condition can be used at $p=p^R$. Then $\Bu$ can be numerically recovered via
	\begin{equation}\label{p>0}
		\Bu(t)= \frac{1}{1-e^{-p^R}}\int_0^{p^R} \Bw(t,p)\, dp, 
	\end{equation}
	or 
	\begin{equation}\label{p>0-1}
	\Bu(t)=e^{p} \Bw(t,p) \qquad {\text{ for any}}\quad  0<p<p^R.
	\end{equation}

	However, when $\lambda_i(H)>0$ for some $i$, some of the waves evolving through \eqref{eq:up} will instead propagate from left to right. If we were solving the problem in domain $p\in [0, \infty)$, we would need a boundary condition at $p=0$, which is not possible.
	Instead, we extend the computational  domain to also include  $[-p^L, 0)$ with   a zero incoming boundary  at $p=-p^L$ for $p^L>0$ sufficiently large (so $e^{-p^L} \approx 0$). 
	In summary, we use zero boundary condition for \eqref{eq:up} restricted on the finite domain $[-p^L, p^R]$.
	It is worth noting that extending to the $p<0$ domain is done for the convenience of employing the Fourier spectral method, which requires periodic boundary conditions, and our zero boundary conditions can be regarded, approximately, as periodic boundary condition.  It is important to note that any waves--corresponding to positive eigenvalues of $H$-- that start from $p<0$ and travel to the right will induce spurious solutions to the domain $p>0$, therefore, when recovering $u$, one needs to select the correct domain to avoid using these spurious waves.  Therefore,  the following theorem \cite{JLM24} must be utilized.

\begin{theorem}\label{thm:recovery u}
	Assume the eigenvalues of $H$ satisfy \eqref{eq:eigenvalues H1}, then the  solution of 
	\eqref{eq:ODE1} can be recovered by
	\begin{equation}\label{eq:recover by one point}
		\Bu(t) = e^{p} \Bw(t,p),\quad \text{for}\;\text{any}\; p> p^{\Diamond},
	\end{equation}
	where $p^{\Diamond}= \max\{\lambda_n(H) t,0\}$, or recovered by using the integration,
	\begin{equation}\label{eq:recover by quad}
		\Bu(t) = e^{p}\int_{p}^{\infty} \Bw(t,q)\;\d q,\quad 
		p> p^{\Diamond}.
	\end{equation}
\end{theorem}

Since \eqref{eq:up} is a {\it hyperbolic} system, thus the initial value problem is {\it well-posed} and one can solve it numerically in a stable way (as long as suitable numerical stability condition--the CFL condition--is satisfied). Thus although the original system \eqref{eq:ODE1} is ill-posed, we can still solve the well-posed problem \eqref{eq:up}, and then recover $\Bu$ using Theorem \ref{thm:recovery u}, as long as $\lambda_n(H)<\infty$! 

	We also point out here that the discretization of the $p$-derivative in Eq. \eqref{eq:up}  
	can be achieved not only by using Fourier spectral methods but also through other numerical schemes, such as finite difference or finite element methods.
	In such cases one can just solve the problem in $p>0$ and, if $\lambda_i(H)>0$ for some $i$, one needs to supply boundary condition at $p=0$ for the  characteristic fields corresponding to positive eigenvalues of $H$.  This boundary data can be set arbitrarily, since the induced spurious wave propagating  with speed $\lambda_i(H)>0$ will not be used when  we recover $u$ using data from $p>p^\diamond$, as laid out  by Theorem \ref{thm:recovery u}.


\section{The backward heat equation}\label{sec:The backward heat equation}

In this section, we consider the following one-dimensional backward heat equation in the infinite domain,
\begin{equation}\label{eq:backward heat Eq}
\begin{aligned}
	\partial_t u &= \partial_{xx} u   \;\, \quad \;\text{in}\;\Omega_x  \quad T>t\geq 0, \\
	u(T,x) &= u_T(x) \quad \;\text{in}\;\Omega_x ,
\end{aligned}
\end{equation}
where $\Omega_x=(-\infty,\infty)$. We want to determine $u(t,\cdot)$ for $T>t\geq 0$ from the data $u_T$.  
The change of variable $v(t,x) = u(T-t,x)$ leads to the following formulation of \eqref{eq:backward heat Eq}, 
\begin{equation}\label{eq:changed backward}
\begin{aligned}
	\partial_t v & = -\partial_{xx} v \;\, \quad \; &&\text{in}\; \Omega_x \quad 0\leq t<T,\\
	v(0,x) &= u_T(x)\quad \; &&\text{in} \; \Omega_x.
\end{aligned}
\end{equation}
It is easy to see that the eigenvalues of $H$ in \eqref{eq:ODE1} are positive after 
spatial discretization.
Let $\hat{g}(\eta)$ denote the Fourier transform of $g(x)\in L^2(\Omega_x)$ and define it by 
\begin{equation*}
\hat{g}(\eta) := \mathscr{F} (g)=\frac{1}{2\pi}\int_{\bbR} e^{ix\eta} g(x)\d x,  \qquad
g(x):=\mathscr{F}^{-1}(\hat g) = \int_{\bbR}
e^{-ix\eta} \hat g(\eta)\d \eta.
\end{equation*}
Let $\|g\|_{H^s(\Omega_x)}$ denote the Sobolev norm  $H^s(\Omega_x)$ defined as  
\begin{equation*}
\|g\|_{H^s(\Omega_x)}:=\big(\int_{\bbR}|\hat g(\eta)|^2(1+\eta^2)^s\;\d\eta\big)^{\frac 12}.
\end{equation*}
When $s=0$, $\|\cdot\|_{H^0(\Omega_x)}$ denotes the $L^2(\Omega_x)$-norm,  and $s$ can be noninteger \cite{evan16}.
For the above solution to be well-defined, we make the following assumptions of $u_T$ for analysis:
\begin{itemize}
	\item[\textbf{(H1)}] Assume ${u}_T\in H^s_0$, namely $H^s$ with compact support. 
	This yields a well-posed problem in a subspace of $L^2$  in the sense of Hadamard \cite{Miranker61}. 
	\item[\textbf{(H2)}]
	For more general $\hat{u}_T$, we will truncate the domain in $\eta$ and begin with the truncated--thus compactly supported--final data. 
\end{itemize}
We remark here that truncation is inherently achieved in our methodology (see $\textbf{(S1)}$ and $\textbf{(S2)}$), thus providing a form of {\it regularization} to the original ill-posed problem.

If the solution of \eqref{eq:backward heat Eq} in this Sobolev space exists, then it must be unique \cite{evan16}. We assume $u(t,x)$ is the unique solution of 
\eqref{eq:backward heat Eq}.
Applying the Fourier transform, one gets the exact solution $u(t,x)$ of problem \eqref{eq:backward heat Eq}:
\begin{equation}\label{eq:exact ut}
u(t,x) = \int_{\bbR} e^{-ix\eta} e^{\eta^2(T-t)}\hat{u}_T\;\d \eta.
\end{equation}
Denoting $u_0(x) = u(0,x)$,  then it is easy to see from \eqref{eq:exact ut} that 
\begin{equation}\label{assum:u0}
\|u_0\|_{H^s(\Omega_x)} = \big(  \int_{\bbR} |\hat u_0|^2 (1+\eta^2)^s\;d\eta\big)^{1/2}<\infty,\quad \text{for}\; s\geq 1.
\end{equation} 

The Schr\"odingerisation for \eqref{eq:backward heat Eq} gives
\begin{equation}
\begin{aligned}\label{w-heat}
	\frac{\D}{\D t} w &= - \partial_{xxp}w  \;\;\quad \quad \text{in}\; \Omega_x \times \Omega_p \quad T>t\geq 0, \\
	w(T,x,p) &= e^{-|p|}u_T(x)   \;\;\quad \text{in}\; \Omega_x \times \Omega_p,
\end{aligned}
\end{equation}
where $\Omega_p =(-\infty,\infty)$.
Again using the Fourier transform technique to \eqref{w-heat} with respect to the variables $p$ and $x$, one gets
the Fourier transform $\hat w(t,\eta,\xi)$ of the exact solution $w(t,x,p)$ of problem~\eqref{w-heat} satisfying
\begin{equation}\label{eq:Fourier transform w}
\frac{\D}{\D t} \hat w  = -i\xi\eta^2 \hat{w},\quad  \hat{w}(T,\eta,\xi) = \hat w_T =\frac{1}{2\pi}\int_{\bbR} \frac{e^{ix\eta} u_T(x)}{\pi(1+\xi^2)} \;\d x = \frac{\hat u_T(\eta)}{\pi (1+\xi^2)} .
\end{equation}
The solution of $w$ is then found to be 
\begin{equation}\label{eq:w}
w(t,x,p) = \int_{\bbR}\int_{\bbR} e^{-i (x\eta + p\xi)} e^{-i\xi\eta^2(t-T)}\hat w_T\;\d\eta \d \xi.
\end{equation}
Note the initial-value problem  \eqref{w-heat} is well-posed,  {\it even for $u_T\in H^s$ not being compactly supported in the Fourier space},  as can be easily seen from \eqref{eq:Fourier transform w}, which is an oscillatory ODE with purely imaginary spectra. Following the proof in \cite[Theorem 5, section 7.3]{evan16}, one has the regularity of $w(0,x,p)\in H^s(\Omega_x)\times H^1(\Omega_p)$ under the assumption in \textbf{(H1)} or \textbf{(H2)}.

Using the Fourier transform only on $x$ gives
\begin{equation}\label{w-tilde}
	\partial_t \hat{w}^x-\eta^2 \partial_p \hat{w}^x=0\quad  T>t\geq 0, \quad \hat{w}^x(T,\eta,p) = e^{-|p|}\hat{u}_T(\eta), 
\end{equation}
where $\hat{w}^x$ is the Fourier transform of $w$ on the $x$-variable, i.e. 
\begin{equation}\label{eq:hat wx l}
	\hat{w}^x(t,\eta,p) = \frac{1}{2\pi }\int_{\bbR} w(t,x,p) e^{i x\eta}\d x.
\end{equation}
Here $\hat{w}^x$ is a linear wave moving to the right, 
if one starts from $t=T$ and going backward in time (one can also use the change of variable $t \mapsto T-t$  to make the problem \eqref{w-heat} forward in time for notation comfort),
and the solution is computed by the inverse Fourier transformation 
\begin{equation}\label{eq:w2}
	w(t,x,p) = \int_{\bbR} e^{-|p+\eta^2(t-T)|} \hat{u}_T(\eta) e^{-ix\eta}\;\d \eta.
\end{equation}
This corresponds to the case of \eqref{eq:up} in which all eigenvalues of $H$ are {\it positive}. 
Therefore, to recover $u(0,x)$ from Eq.\eqref{eq:w2}, one needs to choose  $p> \eta^2 T$ for time $T>0$  to obtain 
\begin{equation*}
	u(0,x) = e^{p} \int_{\bbR} e^{-|p-\eta^2 T|}\hat{u}_Te^{-ix\eta} d\eta = \int_{\bbR} e^{\eta^2 T} \hat u_T e^{-ix\eta}\, d\eta,\quad  p>\eta^2 T.
\end{equation*}

This is the basis for the introduction of our stable computational method:   we solve the $w$-equation \eqref{w-heat} with a suitable--stable-- computational method, and then recover $u$ by either integration or pointwise evaluation using Theorem \ref{thm:recovery u}.

However, in the continuous space $\eta$ maybe unbounded, thus the condition $p>\eta^2 T$ may not be satisfied.  We consider two scenarios here:

\begin{enumerate}
\item[\textbf{(S1)}] The input data $\hat{u}_T$ has compact support, namely $\hat{u}_T=0$ for $|\eta|>\eta_{0}$ for some $\eta_{0}<\infty$.  While this is not generally true, when $u_T(x)$ is smooth in $x$, its Fourier mode decays rapidly, so one can truncate $\eta$ for $|\eta|>\eta_{0}$ with desired accuracy. 

\item[\textbf{(S2)}] One discretizes equation \eqref{w-heat} numerically. This usually requires first to truncate the $x$-domain so it is finite, followed by some numerical discretization of the $x-$derivative, for example the finite difference or spectral method.
Then $|\eta|\le \eta_{0}\leq  \mathscr{O}(1/(\Delta x)^2)$  where $\Delta x$ is the spatial mesh size in $x$.
\end{enumerate}

\subsection{Error estimates}

We start by   choosing an $\eta_{\max}>0$ and denote \begin{equation}\label{etamax}
\delta(\eta_{\max})=\big(\int_{\eta_{\max}}^{\infty} (1+\eta^2)^s |\hat{u}_0|^2\; \d\eta \big )^{\frac 12}
\end{equation}
to be a small quantity, which will be chosen with other error terms to meet the numerical tolerance requirement  (see Remark \ref{remark:recovery u}). We define the ``approximate" solution $u_{\eta_{\max}}$ of \eqref{eq:backward heat Eq} as follows:

\begin{definition}
First, define
\begin{equation}\label{eq:truncated w}
	w_{\eta_{\max}} =
	\int_{\bbR}e^{-|p+\eta^2(t-T)|}\hat u_T(\eta)\chi_{\max}e^{-ix\eta}\;\d \eta, 
\end{equation}
where $\chi_{\max}(\eta)= 1$ when $\eta \in [-\eta_{\max},\eta_{\max}]$, otherwise $\chi(\eta) = 0$. 
Then one obtains the approximate solution $u_{\eta_{\max}}$ by
\begin{equation}\label{eq:recover u}
	u_{\eta_{\max}}(t,x,p) = e^{p} w_{\eta_{\max}}(t,x,p) , \quad \text{or}\quad 
	u_{\eta_{\max}}(t,x,p)=e^{p} \int_{p}^{\infty} w(t,x,q) \;\d q,
\end{equation}
with $p\geq p^{\Diamond}=\eta_{\max}^2T$.
\end{definition}


With $u$ defined as above, we have  the following theorem which gives an error estimate to the above approximation. 

\begin{theorem}\label{thm:recovery of u}
	 Suppose $u_T \in H^s(\Omega_x)$ satisfies the assumption in \textbf{(H1)} or \textbf{(H2)}.
	Let $u(t)=u(t,x)$ be the solution \eqref{eq:backward heat Eq}  given by \eqref{eq:exact ut},
	and  $u_{\eta_{\max}}(t)=u_{\eta_{\max}}(t,x,p)$ given by \eqref{eq:recover u} with $p>p^{\Diamond} = \eta_{\max}^2T$ and $\eta_{\max}$ chosen to satisfy \eqref{etamax}.
	Then for any $p>p^{\Diamond}$, $t\in [0,T)$,
	\begin{equation}\label{u-approx}
		\|u(t)-u_{\eta_{\max}}(t)\|_{L^2(\Omega_x)} \le e^{-\eta_{\max}^2 t}\delta(\eta_{\max})/\eta_{\max}^s.
	\end{equation} 
\end{theorem}
\begin{proof}
 If $u_T$ satisfies $\textbf{(H1)}$ or $\textbf{(H2)}$, it is apparent that 
		$\delta(\eta_{\max})$ for some $\eta_{\max}$ in 
		Eq.\eqref{etamax} is well defined. 
	According to Parseval's equality and  \eqref{eq:exact ut}, \eqref{etamax} -- \eqref{eq:recover u}, it follows that 
	\begin{align}
		\|u-u_{\eta_{\max}}\|_{L^2(\Omega_x)} &= \|\hat u-\hat u_{\eta_{\max}}  \|_{L^2(\Omega_\eta)}
		=\|e^{\eta^2(T-t)}\hat u_T-e^{\eta^2(T-t)}\hat u_T\chi_{\max}\|_{L^2(\Omega_\eta)}  \notag \\
		& =\left(\int_{|\eta|>\eta_{\max}} |e^{\eta^2(T-t)}e^{-\eta^2 T}\hat{u}_0|^2\;\d x\right)^{\frac 12}\notag \\
		& =\left(\int_{|\eta|>\eta_{\max}} |e^{-\eta^2 t} \hat{u}_0|^2\;\d x\right)^{\frac 12}\notag \\
		&= \left(\int_{|\eta|>\eta_{\max}} \frac{ e^{-\eta^2 t}|\hat u_0|^2}{(1+\eta^2)^{s}} (1+\eta^2)^s\;\d \eta\right)^{\frac 12} \notag \\
		&\le \left(\int_{|\eta|>\eta_{\max}} \frac{ e^{-\eta_{\max}^2 t}|\hat u_0|^2}{\eta^{2s}}  (1+\eta^2)^s\;\d \eta\right)^{\frac 12}
		\leq \frac{e^{-\eta_{\max}^2 t}\delta(\eta_{\max})}{\eta_{\max}^s}. \label{eq:u-ueta}
	\end{align}
	The proof is completed.
\end{proof}

\begin{remark} Since we are solving the problem \eqref{eq:backward heat Eq} {\it backward in time}, the error estimate obtained in Theorem \ref{thm:recovery of u} needs to be understood also backward in time, namely the error will {\it increase} as $t$ goes from $T$ to $0$. Numerical experiments in section \ref{sec:Approximation of a piece-wise smooth solution} also confirms this. 
\end{remark}

\begin{remark}\label{remark:recovery u}
	In order to bound the error by $\varepsilon$  at $t=0$,
	the shape estimate of $\eta_{\max}$ from \eqref{eq:u-ueta}
	is that $\eta_{\max}$ is chosen such that 
	\begin{equation}\label{eq:eta_max_choice0}
		\left(\int_{|\eta|>\eta_{\max}} |\hat{u}_0|^2\;\d x\right)^{\frac 12} \leq \varepsilon.
	\end{equation}
	If $u_0\in H^s(\Omega_x)$, one could choose $\eta_{\max}$  satisfying 
	\begin{equation}\label{eq:eta_max_choice}
		\eta_{\max}\geq \left(\frac{\delta}{\varepsilon}\right)^{\frac{1}{s}}.
	\end{equation}
	Considering  $u_0\in H^s(\Omega_x)$, one may have
	$|\hat{u}_0|^2(1+\eta^2)^s = \mathscr{O}(\frac{1}{\eta^{1+\epsilon_0}})$ with $0< \epsilon_0\leq 2$,
	and then $\delta(\eta_{\max}) = \mathscr{O}(1/(\eta_{\max}^{\epsilon_0/2}))$ is small.
	Besides, if $u_0$ is smooth enough, $p^{\Diamond} = (\frac{\delta}{\varepsilon})^{\frac{2}{s}} T$ is not necessarily large. 
	
	Specifically, if the Fourier transform of $u$ has compact support such that $\hat{u}_T(\eta) = 0$ when $\eta>\eta_{0}$, it yields from \eqref{eq:u-ueta}-\eqref{eq:eta_max_choice0} that there exits no error of recovery by choosing $p^{\Diamond} = \eta_{0}^2T$.  Take an example, $u = \alpha(t)\cos(\omega_0 x)$ (or $u = \alpha(t)\sin(\omega_0 x)$) is a periodic function in $[\frac{-\pi}{\omega_0},\frac{\pi}{\omega_0}]$. 
	We extend it periodically over the whole field of real domain $\bbR$. The Fourier transform of $u$ is $\hat{u}(t,\eta) = \alpha(t) \pi\big(\delta(\eta-\omega_0) + \delta(\eta+\omega_0)\big)$,
	where $\delta(\eta+\omega_0)$ gives an impulse that is shifted to the left by $\omega_0$, likewise the function $\delta(\eta-\omega_0)$ yields an impulse shifted to the right by $\omega_0$. In this case, $\hat{u}(t,\eta)\equiv 0$ when $|\eta|>\omega_0=\eta_0$.
	Therefore, the best choice of $p^{\Diamond}$ for 
	$u=\alpha(t)\cos(\omega_0 x)$ (or $u = \alpha(t)\sin(\omega_0 x)$) is 
	\begin{equation}\label{eq:choice of pDiamond}
		p^{\Diamond} = \omega_0^2 T = \eta_0^2 T.
	\end{equation}  
	At this point,  $\eta_{\max}=\omega_0=\eta_0$,  one has $\delta(\eta_{\max}) = 0$,  $u=u_{\eta_{\max}}$ for $p> p^{\Diamond}$.
\end{remark}

In practice, the data at $t=T$  is usually imprecise since it depends on the reading of physical 
measurements. Consider a perturbed data  $u_T^{\zeta}$, 
which is a small disturb of $u_T$.  The  Fourier transform of $u_T^\zeta$ may  not decay as $|\eta|\to \infty$, which leads to the severely ill-posed problems of \eqref{eq:backward heat Eq}. 
However, after Schr\"odingerisation, a disturb of the Fourier transform of $w$ defined by $\hat w_T^{\zeta} = \frac{\hat u_T^{\zeta}(\eta)}{\pi (1+\xi^2)}$ does not affect the well-posedness of $w$ even if $w_T^{\zeta}$ is merely in $L^2(\bbR)$ .
Assume the measured data $u_T^{\zeta}$ satisfies 
\begin{equation}
	\|u_T^{\zeta} - u_T\|_{L^2(\Omega_x)}\leq \zeta_0. 
\end{equation} 
Define the approximate solution at time $t$ from $u_T^{\zeta}$ by
\begin{equation}\label{eq:w_delta}
	\begin{aligned}
		w_{\eta_{\max},\zeta} &=  \int_{\bbR}\int_{\bbR} e^{-i (x\eta + p\xi)} e^{-i\xi\eta^2(t-T)}\hat w_T^{\zeta}\chi_{\max}\;\d\eta \d \xi\\
		&=\int_{\bbR} e^{-|p+\eta^2(t-T)|} \hat{u}_T^{\zeta}(\eta) e^{-ix\eta} \chi_{\max}\;\d \eta.
	\end{aligned}
\end{equation}
Now  we have the approximate solution $w_{\eta_{\max},\zeta}$.

\begin{theorem}\label{thm:recover of u with input noise}
	 Suppose $u_T \in H^s(\Omega_x)$ satisfies the assumption in \textbf{(H1)} or \textbf{(H2)}.
	Let $u(t,x)$ be given by \eqref{eq:backward heat Eq} and $ w_{\eta_{\max},\zeta}(t,x,p)$ given by \eqref{eq:w_delta}, respectively.
	If $u_{\eta_{\max},\zeta}$ is recovered by 
	\begin{equation}\label{eq:recover u_delta}
		u_{\eta_{\max},\zeta} = e^{p} w_{\eta_{\max},\zeta}(t,x,p),  \quad \text{or}\quad  u_{\eta_{\max},\zeta} =e^{p} \int_{p}^{\infty} w_{\eta_{\max},\zeta}(t,x,q) \;\d q, 
	\end{equation}
	for any $p>p^{\Diamond} = \eta_{\max}^2 T$, and $\eta_{\max}$ is chosen such that 
	$ \frac{\delta(\eta_{\max})}{|\eta_{\max}|^s}
	+ e^{\eta_{\max}^2 T}\zeta_0 \leq \varepsilon$,
	then 
	\begin{equation*}
		\|u(t)-u_{\eta_{\max},\zeta}(t)\|_{L^2(\Omega_x)} 
		\leq \varepsilon.
	\end{equation*} 
\end{theorem}

\begin{proof}
	From the definition of $u_{\eta_{\max},\zeta}$ in \eqref{eq:recover u_delta}, one has
	\begin{equation*}
		u_{\eta_{\max},\zeta} = \int_{\bbR} e^{-ix\eta} e^{\eta^2(T-t)}\hat u_T^{\zeta}(\eta)\chi_{\max}\;\d \eta,
	\end{equation*}
	for $p>p^{\Diamond}$.
	Then one has 
	\begin{equation}\label{eq:u_delta_etamax}
		\|u_{\eta_{\max},\zeta} - u_{\eta,\max}\|_{L^2(\Omega_x)}
		=\big(\int_{|\eta|\leq \eta_{\max}} |e^{\eta^2(T-t)} (\hat u_T^{\zeta} - \hat u_T)|^2\;\d\eta\big)^{\frac 12} \leq e^{\eta_{\max}^2(T-t)} \zeta_0.
	\end{equation}
	The proof is finished by combining \eqref{eq:u-ueta} and  \eqref{eq:u_delta_etamax}.
\end{proof}

\subsection{Variable coefficient problems}
	When the system has a spatially varying coefficient, the Fourier transform cannot be applied, and the analysis in Theorem~\ref{thm:recovery of u} no longer holds. However, the Schr\"odingerisation still works.

We consider variable coefficient equation 
\begin{equation}\label{eq:vari coef BHEQ} 
	\partial_t u=\partial_x(a(x)\partial_x u) \quad x\in \Omega_x, \quad  0<t<T, \quad  u(T, x)=u_T(x),
\end{equation}
with Dirichlet's boundary condition of bounded domain $\Omega_x$ in straightforward way,
where $a(x)\in C^{\infty}(\bar\Omega_x)$ and $a(x)\ge \alpha_0 >0$.

Let $\phi_k$ be the $k$-th eigenfunction corresponding to $\lambda_k$ of the operator $\mathcal{L} := -\partial_x( a(x)\partial_x \cdot )$ such that 
\begin{equation}\label{eq:eigenfun of L}
	\mathcal{L} \phi_k = \lambda_k \phi_k, \quad \phi_k|_{\partial \Omega_x} =0.
\end{equation}
It is easy to see that $\phi_k \in C^{\infty}(\Omega_x)$ from the regularity of the second-order elliptic equations \cite[Theorem 3, section~6.3]{evan16}.
They form an orthonormal basis set of $L^2(\Omega_x)$. According to standard theory on the eigenvalues of symmetric elliptic operators \cite[Theorem 1, section 6.5]{evan16}, one has 
\begin{equation}
	0<\lambda_1\leq \lambda_2\leq \lambda_3\leq \cdots, \quad 
	\lambda_k\to \infty \quad \text{as}\; k\to \infty.
\end{equation}
We search for solution to \eqref{eq:vari coef BHEQ} at time $t$ from the data $u_T$ in the form 
\begin{equation}\label{eq:u expansion of eigenFun}
	u(t,x) = \sum_{k\in \bbN^+} \alpha_k(t)\phi_k(x) = \sum_{k\in \bbN^+} e^{\lambda_k (T-t)} \alpha_k(T)\phi_k(x) \quad T>t\geq 0,
\end{equation}
where  $\alpha_k(T) = \int_{\Omega_x} u_T\phi_k\;\d x$. 

Define the equivalent norm of $\|\cdot\|_{H^s}$, $s\in \bbN$, again denoted by $\|\cdot\|_{H^s}$:
\begin{equation}\label{eq:Hs vare coef u}
	\|v\|_{H^s(\Omega_x)}^2 = \sum_{k\in \bbN^+} |\beta_k|^2(1+\lambda_k + \cdots \lambda_k^s),
\end{equation} 
where $\beta_k = \int_{\Omega_x} v\phi_k\;\d x$.
Here $\|\cdot\|_{H^0(\Omega_x)}:=\|\cdot \|_{L^2(\Omega_x)}$ denotes  $L^2(\Omega_x)$ norm when $s=0$,
and $s$ can also be a noninteger for $1\leq s\leq 2$ \cite[Proposition 2.1]{KR97}.
For  the solution to \eqref{eq:vari coef BHEQ} to be well-defined, we assume 
the expansion of $u_T\in H^s(\Omega_x)$ to the orthonormal basis has only finite number $n_{\max}$ of terms, leading to $u(t,x) \in H^s(\Omega_x)$ for all $0\le t\le T$,  i.e.
\begin{equation}\label{eq:regularity of u0}
	\begin{aligned}
		\|u(t, 
		\cdot)\|_{H^s(\Omega_x)}^2 &= \sum_{k\le n_{\max}} |\alpha_k(t)|^2(1+\lambda_k+\cdots\lambda_k^s)  \\
		&= \sum_{k\le n_{\max}} |e^{\lambda_k (T-t)}\alpha_k(T)|^2(1+\lambda_k+\cdots\lambda_k^s)< \infty.
	\end{aligned}
\end{equation}
For more general cases, we  truncate to the finite terms of expansions of the input data. This regularizes the originally ill-posed problem.  
Next, assume $w(t,x,p)$ is in the form of 
\begin{equation*}
	w(t,x,p) = \sum_{k\in \bbN^+} w_k(t,p)\phi_k(x).
\end{equation*}
We correspondingly project the input data $w(T,x,p)$ on the same basis set and Eq.~\eqref{w-heat}
is equivalent to 
\begin{equation*}
	\frac{\D}{\D t } w_k(t,p) =-\lambda_k \partial_p w_k(t,p),\quad 
	w_k(T,p) = e^{-|p|} \alpha_k(T).
\end{equation*}
Note $w$ is well-defined even if the series is not truncated. We then define the truncated approximate solution by
\begin{equation}\label{eq:wnmax}
	w_{n_{\max}} =\sum_{1\leq k\leq n_{\max}} e^{-|p-\lambda_k (T-t)|} \alpha_k(T) \phi_k(x).
\end{equation}
Let
\begin{equation}
	\delta(n_{\max}) = \left(\sum_{k=n_{\max}}^{\infty} |\alpha_k(0)|^2\lambda_k^s\right )^{1/2},\quad 
\end{equation}
we choose $n_{\max}$ such that
\begin{equation}\label{eq:nmax}
	\delta(n_{\max})/\lambda_{n_{\max}}^{\frac{s}{2}} \leq \varepsilon.
\end{equation}
for the desired precision $\varepsilon$.
We remark here $n_{\max}$ exists 
from \eqref{eq:regularity of u0} if $u_T\in H^s(\Omega_x)$.

\begin{theorem}
	Assume $u_T\in H^s(\Omega_x)$ with only finite number terms of expansion  to the orthonomal basis, or if not, truncated to be so. 
	Let $u(t,\cdot)$ be the solution to \eqref{eq:vari coef BHEQ} given by \eqref{eq:u expansion of eigenFun} 
	and $w_{n_{\max}}(t,x,p)$ be given by \eqref{eq:wnmax},
	respectively. 
	If $u_{n_{\max}}$ is recovered by 
	\begin{equation}
		u_{n_{\max}} = e^p w_{n_{\max}}\quad  \text{or}\quad 
		u_{n_{\max}} = e^p\int_p^{\infty} w_{n_{\max}}(t,x,q)\;\d q,
	\end{equation}
	where $p>p^{\Diamond} = \lambda_{n_{\max}}T$, and $n_{\max}$ is chosen by \eqref{eq:nmax}. Then 
	\begin{equation}
		\|u(t,\cdot) - u_{n_{\max}}(t,\cdot)\|_{L^2(\Omega_x)}\leq e^{-\lambda_{n_{\max}}t}\delta(n_{\max})/\lambda_{n_{\max}}^{\frac{s}{2}}\leq  \varepsilon.
	\end{equation}
\end{theorem}

The proof is similar to Theorem~\ref{thm:recovery of u} by using the spectral theory of the symmetric elliptic operators and the norm defined in \eqref{eq:Hs vare coef u}. We omit it here.

\section{Discretization of the Schr\"odingerised equation} \label{sec:discretisation}

In this section, we consider the discretization of \eqref{w-heat} and the corresponding error estimates. 
For simplicity, we consider $u$ to be a periodic function defined in $\Omega_x=[-\pi ,\pi ]$.
\subsection{The numerical discretization}\label{sec:quantum algorithms for the backward Eq}
We discretize the $x$ domain uniformly  with the  mesh size $\triangle x = 2\pi /M$, where $M$ is a positive even integer and the grid points are denoted by $-\pi =x_0<\cdots<x_{M}=\pi$.
The $1$-D basis functions for the Fourier spectral method are usually chosen as
\begin{equation} \label{eq:phi nux}
	\phi_j(x) = e^{i\mu_j (x+\pi )},\quad \mu_j =  (j-M/2),\quad j\in [M]. 
\end{equation}
Considering the Fourier spectral discretization on $x$, one easily gets
\begin{equation}\label{eq:tilde uh}
	\frac{\D}{\D t} \Bu_h = - \Phi A (\Phi)^{-1}\Bu_h,\quad 
	\Phi = (\phi_{ij})_{M\times M} = (\phi_j(x_i))_{M\times M},
\end{equation}
where $\Bu_h(T)= [u_T(x_0),u_T(x_1),\cdots,u_T(x_{M-1})]^{\top}$ and $A=D_x^2$. 
The matrix $D_x$ is obtained by $D_{x} = \text{diag}\{\mu_0,\cdots,\mu_{M-1}\}$. By applying matrix exponentials, one has
\begin{equation}\label{eq:tilde uh At}
	\Bu_h(t) = \Phi e^{A(T-t)} (\Phi)^{-1}\Bu_h(T).
\end{equation}
However,  it is difficult to obtain the numerical solution $ \Bu_h$ in the classical computer from \eqref{eq:tilde uh At} since the problem is unstable.  

We now introduce the discretization of $p$ domain as in \cite{JLM24}.	
First truncating  the $p$-region to $[-\pi L,\pi L]$, where   $\pi L$ is large enough such that 
	\begin{equation}\label{eq:require pi L}
		e^{-(\pi L -4.5\eta_{\max}^2 T)} \|u_T\|_{L^2(\Omega_x)} \leq \mathscr{O}(\varepsilon),
	\end{equation}
	with $\varepsilon$ error bound.
	We point out that the constant $4.5$ is just for numerical analysis which is quite mild for numerical tests, and it is sufficient to take it as approximately $1$.  

Then, we can treat $w$ at the boundary  
as $w(t,x,\pi L)\approx w(t,x,-\pi L)\approx 0$. 
Using the spectral method, one gets the transformation and difference matrix
\begin{equation*}
	(\Phi_p)_{ij} = (e^{i\mu^p_i(p_j+\pi L)})\in \bbC^{N\times N}, \quad 
	D_p = \text{diag} \{\mu^p_0,\mu^p_1,\dots,\mu^p_{N-1}\},
	\quad \mu^p_i = (i-N/2)/L,
\end{equation*}
where $p_j = -\pi L +j\triangle p$ with $\triangle p = \frac{2\pi L}{N}$.
Applying the discrete Fourier spectral discretization on $p$ and $x$, i.e.
$\Phi_{px} = \Phi^p \otimes \Phi$, 
it yields 
\begin{equation}\label{eq:tilde w}
	\frac{\D}{\D t} \Bw_h = i \Phi_{px}  (D_p \otimes A)\Phi_{px}^{-1} \Bw_h, \quad 
	\Bw_h (T) = 
	\Bg_h \otimes \Bu_h(T),
\end{equation}
where $\Bg_h =[e^{-|p_0|},e^{-|p_1|},\cdots,e^{-|p_{N-1}|}]^{\top}$. 
	The matrices $\Phi^{-1}$ ($\Phi$) or $\Phi_p^{-1}$ ($\Phi_p$) can be implemented by (inverse) fast  Fourier transforms (FFT).  Since $D_p\otimes A$ is  a diagonal matrix, the solution at time $t$ is computed by 
	\begin{equation}\label{eq:unitary wht}
		\Bw_h(t) = \Phi_{px}\Cu(T-t)\Phi_{px}^{-1} \Bw_h(T), \quad \Cu(T-t) = \exp(i D_p\otimes A(T-t)).
	\end{equation}
	This approach ensures that no error is introduced in the time discretization. 
	The numerical solution to the system \eqref{eq:backward heat Eq} is 
	\begin{equation}
		\Bu_h = e^{p_{k_0}}((\Be_{k_0}^{(N)})^{\top}\otimes I_M) 
		\Bw_h,
	\end{equation}
	where $k_0=\min \{k:p_k> p^{\Diamond}=\eta_{\max}^2 T\}$ from Theorem~\ref{thm:recovery of u}, or can be recovered by second order numerical integration
	\begin{equation}
		\Bu_h  = \frac{\sum_{j\in \Cj} \triangle p ((\Be_{j}^{(N)})^{\top}\otimes I_M) 
			\Bw_h}{e^{-(p_{j_0}-0.5\triangle p)}-e^{-(\pi L-0.5\triangle p)}},  
	\end{equation}
	with $j_0 = \min \Cj = \min \{j:p_j-0.5\triangle p> p^{\Diamond}=\eta_{\max}^2 T\}$.

\subsection{Error analysis for the spatial discretization}
Following the general error estimates of  Schr\"odingerisation in the extended domain in \cite{JLM24}, 
we derive  the specific estimates for the above approximation to the backward heat equation under the assumption that $u_T\in H^s(\Omega_x)$ satisfies \textbf{(H1)} or \textbf{(H2)}.

Define the complex $(N+1)$ and $(M+1)$-dimensional space with respect to $p$ and $x$, respectively
\begin{equation}
	X_N^p = \text{span}\{e^{ik(p/L)}:-\frac{N}{2}\leq k\leq \frac{N}{2}\},\quad 
	X_M^x = \text{span}\{e^{ilx}:-\frac{M}{2}\leq l\leq \frac{M}{2}\}.
\end{equation}
The approximation of $w$ in the finite space $X_N^p\times X_M^x$ from the numerical solution in Eq. \eqref{eq:tilde w} is 
\begin{equation}\label{eq:whd(t,x,p)}
	\begin{aligned}
		w_h^d(t,x,p) =  \sum_{|k|\leq \frac{N}{2}} \sum_{|l|\leq \frac{M}{2}}\tilde w_{k,l} e^{ik(\frac{p}{L}+\pi)}e^{il(x+\pi)}, \quad
		\tilde w_{k,l} = (\Be_{j_k}^{(N)}\otimes \Be_{j_l}^{(M)})^{\top} \Phi_{px}^{-1}\Bw_h(t),
	\end{aligned}
\end{equation}
where  $j_k = -k+N/2$, $j_l = -l+M/2$. Correspondingly, the approximation of $u$ is defined by 
\begin{equation}\label{eq:uhd}
	u_h^d(t,x,p) = e^{p} w_h^d(t,x,p),\quad \text{or} \quad 
	u_d^*(t,x,p) = \frac{1}{e^{-p}-e^{-\pi L}} \int_{p}^{\pi L} w_h^d(t,x,q)\;\d q,
\end{equation}
where $p>p^{\Diamond}=\eta_{\max}^2 T$.

	It is apparent that the discretization serves as an approximation for the following system with periodic boundary conditions:
	\begin{equation}\label{w-heat periodic}
		\begin{cases}
			\frac{\D}{\D t} \Cw  = -\partial_{xxp} \Cw\quad \quad \;  \text{in}\; \Omega_x\times \Omega_p \quad T> t \geq 0,\\
			\Cw(t,x,-\pi L) =  \Cw(t,x,\pi L), \quad 
			\Cw(t,-\pi,p) =  \Cw(t,\pi,p), \\
			\Cw(T,x,p)  = \Cg(p) u_T(x),
		\end{cases}
	\end{equation}
	where $\Omega_p = (-\pi L,\pi L)$, and $\Cg$ is periodic with a period of $2\pi L$, such that 
	\begin{equation}\label{def:Cg}
		\Cg (p)= e^{-|p-2m\pi L|}\quad \quad (2m-1)\pi L \leq p <(2m+1)\pi L, \quad m\in \bbZ.
	\end{equation}
	Since $\Cw$ is also periodic  with respective to $x$, 
	the Fourier transform on variable $x$ of $\Cw$ is defined by
	\begin{align}
		\hat{\Cw}^x(t,\eta,p)  =  \frac{1}{2\pi} \int_{\bbR} e^{ix \eta} \Cw(t,x,p)\;\d x
		\label{eq:hat Cwx l}.
	\end{align}
	According to the fundamental results on regularity of transport equations, it is well known that the initial value problem \eqref{w-heat periodic} is well posed, and 
	$\Cw(0,x,p)\in H^s(\Omega_x)\times H^1(\Omega_p)$ satisfies 
	\cite[Theorem 5, section 7.3]{evan16}
	\begin{equation}
		\|\Cw(0,x,p)\|_{H^s(\Omega_x)\times H^1(\Omega_p)} \lesssim 
		\|\Cg\|_{H^1(\Omega_p)}\|u_0\|_{H^s(\Omega_x)} \lesssim \|w_0\|_{H^s(\Omega_x)\times H^1(\Omega_p)}.
	\end{equation} In addition,
	the standard estimate of spectral methods shows
	\begin{align}
		\|w_h^d - \Cw(0,x,p)\|_{L^2(\Omega_p)\times L^2(\Omega_x)} &\lesssim (\triangle p+\triangle x^s) \|\Cw(0,x,p)\|_{H^s(\Omega_x)\times H^1(\Omega_p)} \notag \\
		&\lesssim (\triangle p+\triangle x^s)\|w_0\|_{H^s(\Omega_x)\times H^1(\Omega_p)}. \label{eq:err spectral}
	\end{align}
	\begin{lemma}\label{lem:err Cew BD}
		Suppose $\eta_{\max}$ satisfies \eqref{eq:eta_max_choice},
		$\pi L>3\eta_{\max}^2T$, and $u_T\in H^s(\Omega_x)$ satisfies \textbf{(H1)} or \textbf{(H2)}.
		Let $\Ce_w(t,x,p) = \Cw(t,x,p) - w(t,x,p)$. It follows that 
		\begin{equation*}
			\begin{aligned}
				\int_0^T |\Ce_w(t,\cdot,-\pi L)|_{H^1(\Omega_x)}^2 \d t
				\leq  \frac{5}{2}e^{-2(\pi L-3\eta_{\max}^2T)}\|u_T\|^2_{L^2(\Omega_x)} + e^{-3\eta_{\max}^2 T}\varepsilon^2.
			\end{aligned}
		\end{equation*}
	\end{lemma}
	\begin{proof}
		It is easy to obtain  from the system in Eq.\eqref{eq:w}
		that
		\begin{align}
			\hat{w}^x(t,\eta,p) = e^{-|p+\eta^2(t-T)|}\hat u_T(\eta) = e^{-|p+\eta^2(t-T)|-\eta^2T}\hat{u}_0(\eta). \label{eq:hat wl}
		\end{align}
		Similarly, one has
		\begin{equation}
			\hat{\Cw}^x(t,\eta,p) = \Cg(p+\eta^2(t-T))\hat u_T(\eta)=\Cg(p+\eta^2(t-T))e^{-\eta^2T}\hat{u}_0(\eta).\label{eq:hat Cwl}
		\end{equation}
		Using the semi $H^1$ norm in phase space, it yields 
		\begin{equation*}
			|\Ce_w(t,\cdot, -\pi L)|_{H^1(\Omega_x)}^2 = \int_{\bbR} \eta^2|\hat{w}^x(t,\eta, -\pi L)
			-\hat{\Cw}^x(t, \eta,-\pi L)|^2 \;\d \eta.
		\end{equation*}
		According to \eqref{eq:hat wl} and \eqref{eq:hat Cwl},
		changing the order of integration and letting $q = -\pi L+\eta^2 (t-T)$,
		it yields
		\begin{align}\label{eq:err Cew}
			&\int_0^T 
			| \Ce_w(t,\cdot,-\pi L)|_{H^1(\Omega_x)}^2
			\; \d t  = I_1+I_2,
		\end{align}
		where $I_1$ and $I_2$ are computed by 
		\begin{align*}
			I_1= & \,\int_{|\eta|\leq \sqrt{3}\eta_{\max}}
			\hat{u}^2_T(\eta)\int_{-\pi L -\eta^2 T}^{-\pi L} |e^q - \Cg(q)|^2 \d q  \d \eta, \\
			I_2 = &\int_{|\eta|>\sqrt{3}\eta_{\max}} \hat{u}_T^2(\eta) \int_{-\pi L -\eta^2 T}^{-\pi L}
			|e^q - \Cg(q)|^2 \d q \d\eta.
		\end{align*}
		Since $\Cg(q)$ is periodic with a period of $2\pi L$ defined in \eqref{def:Cg}, $\pi L>3\eta_{\max}^2 T$, and $ \eta_{\max}$ satisfies Eq. \eqref{eq:eta_max_choice},  one has
		\begin{align}
			I_1 &= \int_{|\eta|\leq \sqrt{3}\eta_{\max}}  
			\hat{u}_T^2\int_{-\pi L -\eta^2T}^{-\pi L}
			|e^q - e^{-(q+2\pi L)}|^2 \;\d q \d \eta  \notag \\
			& \leq \frac{1}{2}e^{-2(\pi L -3\eta_{\max}^2T) }\|u_T\|^2_{L^2(\Omega_x)},  \label{eq:I1} \\
			I_2 &= \int_{|\eta|>\sqrt{3}\eta_{\max}} 
			\hat{u}_T^2 \int_{-\pi L-\eta^2 T}^{-\pi L} |e^q-\Cg(q)|^2\d q \d \eta \notag  \\
			&\leq 2\int_{|\eta|>\sqrt{3}\eta_{\max}}  \bigg( 
			\hat u_T^2\int_{-\pi L-\eta^2 T}^{-\pi L} e^{2q} \d q 
			+  
			e^{-2\eta^2 T}
			\hat{u}_0^2 \int_{-\pi L-\eta^2 T}^{-\pi L} |\Cg(q)|^2 \d q \bigg)\d \eta 
			\notag \\
			& \leq 2e^{-2\pi L}\|u_T\|_{L^2(\Omega_x)}^2
			+\int_{|\eta|> \sqrt{3}\eta_{\max}} e^{-\eta ^2 T}\hat{u}_0^2 \d \eta  \notag \\
			& \leq 2e^{-2\pi L} \|u_T\|_{L^2(\Omega_x)}^2 + e^{-3\eta_{\max}^2 T}\varepsilon^2. \label{eq:I2}
		\end{align}
		The proof is completed by inserting \eqref{eq:I1} and \eqref{eq:I2} into \eqref{eq:err Cew}.
	\end{proof}
	\begin{lemma}\label{lem:w0-Cw0}
		Assume the assumptions in Lemma \ref{lem:err Cew BD} hold.  
		With $w(t,x,p)$ and $\Cw(t,x,p)$ defined in \eqref{w-heat} and \eqref{w-heat periodic}, respectively, it follows that 
		\begin{equation}
			\|w_0-\Cw_0\|_{L^2(\Omega_x)\times L^2(\Omega_p)} \lesssim  e^{-(\pi L-3\eta_{\max}^2T)}\|u_T\|_{L^2(\Omega_x)} +e^{-3/2\eta_{\max}^2 T} \varepsilon,
		\end{equation}
		where  $w_0 = w(0,x,p)$ and $\Cw_0 = \Cw(0,x,p)$.
	\end{lemma}
	\begin{proof}
		By introducing $\Ce_w = \Cw-w$, one has 
		\begin{equation}\label{wp-w-heat}
			\frac{\D}{\D t} \Ce_w = -\partial_{xxp}  \Ce_w \quad  
			\Ce_w(T,x,p) = 0\;\; \text{in}\; \Omega_x\times \Omega_p.
		\end{equation}
		Multiplying Eq.\eqref{wp-w-heat} by $\Ce_w$, it follows that
		\begin{equation*}
			\int_0^T \int_{\Omega_{x}\times \Omega_p} \partial_t \Ce_w \Ce_w \d x \d p \d t=
			\int_{0}^T\int_{\Omega_x\times \Omega_p} (-\partial_{xxp} \Ce_w) \Ce_w \d x \d p\d t.
		\end{equation*}
		An integration by parts gives 
		\begin{align}\label{eq:dCewDt}
			\int_0^T \frac{\D}{\D t}\|\Ce_w\|^2_{L^2(\Omega_x)\times L^2(\Omega_p)} \d t= \int_{0}^T |\Ce_w(t,\cdot,\pi L)|_{H^1(\Omega_x)}^2 -|\Ce_w(t,\cdot,-\pi L)|_{H^1(\Omega_x)}^2 \d t. 
		\end{align}
		From Lemma~\ref{lem:err Cew BD} and \eqref{eq:dCewDt}, one obtains  
		\begin{equation*}
			\begin{aligned}
				&\|\Ce_w(0,\cdot,\cdot)\|_{L^2(\Omega_x)\times L^2(\Omega_p)}
				\leq \big(\int_0^T |\Ce_w(t,\cdot,-\pi L)|_{H^1(\Omega_x)}^2 \d t \big)^{1/2}  
				\\
				\lesssim \; &
				e^{-(\pi L -3\eta_{\max}^2 T)} \|u_T\|_{L^2(\Omega_x)} + e^{-3/2\eta_{\max}^2 T}\varepsilon.
			\end{aligned}
		\end{equation*}
		The proof is finished.
	\end{proof}
	\begin{theorem}\label{thm:err of uhd-u}
		Suppose $\eta_{\max}$ satisfies \eqref{eq:eta_max_choice},
		$\pi L$ is large enough to satisfies 
		\eqref{eq:require pi L}, and $u_T\in H^s(\Omega_x)$ satisfies \textbf{(H1)} or \textbf{(H2)}. 
		Define $u_h^d=e^p w_h^d$ with $p\in \tilde \Omega_p= (p^{\Diamond},p^{\Diamond}+\Cm)$, and $p^{\Diamond} = \eta_{\max}^2 T$, $\Cm $ is a length such that $\Cm  \leq \min\{ \pi L -\eta_{\max}^2 T, 1/2\eta_{\max}^2T \}$.
		There holds 
		\begin{equation}
			\|u_0-u_h^d\|_{L^2(\tilde \Omega_p)\times L^{2}(\Omega_x)}\lesssim (\triangle p + \triangle x^s)e^{ \eta_{\max}^2 T+ \Cm }\|w_0\|_{H^s(\Omega_x)\times H^1(\Omega_p)} +\varepsilon.
		\end{equation}
	\end{theorem}
	\begin{proof}
		From the recovery rule in Theorem \ref{thm:recovery of u}, it follows that 
		\begin{equation}
			\|u_0-u_h^d\|_{L^{2}(\tilde \Omega_p)\times L^{2}(\Omega_x)} 
			= \|e^p w_0 - e^pw_h^d\|_{L^{2}(\tilde \Omega_p)\times L^{2}(\Omega_x)}.
		\end{equation}
		Using the triangle equality  and 
		Eq.\eqref{eq:err spectral}, it yields 
		\begin{equation}
			\begin{aligned}\label{eq:err uhd}
				&\|u_0 -u_h^d\|_{L^2(\tilde \Omega_p)\times L^2(\Omega_x)}  \\
				\leq \;&\|e^p \Cw_0- e^p w_h^d\|_{L^{2}(\tilde \Omega_p)\times L^{2}(\Omega_x)}
				+\|e^p w_0 -e^p\Cw_0\|_{L^{2}(\tilde \Omega_p)\times L^2(\Omega_x)}  \\
				\leq \;&(\triangle p + \triangle x^s)e^{\eta_{\max}^2 T + \Cm} \|w_0\|_{H^s(\Omega_x)\times H^1(\Omega_p)}
				+e^{\eta_{\max}^2T+\Cm }\|w_0-\Cw_0\|_{L^2(\Omega_x)\times L^2(\tilde \Omega_p)}.
			\end{aligned}
		\end{equation}
		The proof is completed by using Lemma~\ref{lem:w0-Cw0} , Eq.\eqref{eq:require pi L} and  $\Cm \leq \frac{1}{2} \eta_{\max}^2T$.
	\end{proof}
	
	As can be seen from Theorem \ref{thm:recovery of u}, the error increases as the computation time progresses from $T$ to $0$, which is also reflected in the numerical calculations. The error we provided here corresponds to the error at the moment $t=0$, which is the worst-case scenario. This aligns with the numerical test in section ~\ref{sec:Approximation of a piece-wise smooth solution}, where the error increases as the time approaches zero. Consequently, the error estimation in Theorem ~\ref{thm:err of uhd-u} can be considered as taking the $L^{\infty}$ norm in the time direction.

	We point out that our algorithm offers three significant advantages. Firstly, the Hamiltonian nature of the semi-discrete system \eqref{eq:tilde w} allows for the development of a {\it computationally stable} scheme, provided the CFL condition is met. Secondly, Theorem \ref{thm:recovery of u} ensures a straightforward recovery of the original variable $\Bu_h$ from $\Bw_h$. Although a time-dependent exponential growth factor $e^{\eta_{\max}^2T}$ is introduced at this step, we highlight that the value of $\eta_{\max}$ is relatively small, as noted in Remark~\ref{remark:recovery u}. Furthermore, by choosing an appropriate interval in the extended space to recover the original variables, essentially truncating the Fourier modes of the input data, the originally ill-posed problem becomes well-posed. Although the initial-value problem to the Schr\"odingerised equation is well-posed even without this Fourier truncation, to recover $\Bu_h$ one needs {\it finite} Fourier mode.
	Thirdly, our algorithm provides a quantum algorithm for solving ill-posed problems that can be implemented on a quantum computer (see section \ref{sec:quantum}).


\subsection{A higher order improvement in $p$}\label{higher-p}
The first order convergence in $p$ was due to the lack of regularity in $w$, which is continuous but not in $C^1$. The convergence rate can be improved by choosing a smoother initial data for $w$ in the extended space. For example,  consider the following setup: 
\begin{equation}
	\begin{aligned}\label{w-heat_modified}
		\frac{\D}{\D t} w= - \partial_{xxp}w \quad \text{in}\; \Omega_x \times \Omega_p, \quad T>t\geq 0, \quad 
		w(T)= g(p)u_T(x)\quad \text{in}\; \Omega_x \times \Omega_p,
	\end{aligned}
\end{equation}
where  $g(p)$ defined in $\bbR$ satisfies 
\begin{equation}
	g(p) = \begin{cases}
		h(p)\quad &p\in (-\infty, 0],\\
		e^{-p}\quad &p\in (0,\infty),
	\end{cases}
\end{equation}
with $h(p) \in L^2((-\infty,0))$. It is easy to check 
the Fourier transform of $w$ denoted by $\hat w$ satisfies
\begin{equation*}
	\frac{\D}{\D t} \hat w = -i \xi \eta^2 \hat w,
	\quad 
	\hat{w}(T,\xi,\eta) = \hat{g}(\xi) \hat{u}_T(\eta),
\end{equation*}
where $\hat{g}(\xi) = \mathscr{F}(g)$, and 
$w(t,x,p) = \int_{\bbR} g(p+\eta^2(t-T)) \hat{u}_T(\eta)e^{-i x\eta}\;\d\eta$. In addition, 
the results in section~\ref{sec:The general framework},~\ref{sec:The backward heat equation} still hold, since we do not care the solution when $p<0$. 

From Eq.\eqref{eq:err spectral} and Theorem~\ref{thm:err of uhd-u}, the limitation of the convergence order mainly comes from the non-smoothness of $g(p)$. In order to improve the whole convergence rates, we could smooth $g(p)$ by using higher order polynomials, for example
\begin{equation}
	h(p) = \Cp_{2k+1} (p),\quad p\in [-1,0],   \quad \quad   h(p) = e^{p},\quad p\in (-\infty,-1),
\end{equation}
where $\Cp_{2k+1}$ is a  Hermite interpolation  \cite[section~2.1.5]{stoer} and satisfies  
\begin{equation*}
	\begin{aligned}
		\big(\partial_{p^{\alpha}}\Cp_{2k+1}(p)\big)|_{p=0}&=\big(\partial_{p^{\alpha}}(e^{-p})\big)|_{p=0} =(-1)^{\alpha},\\
		\big(\partial_{p^{\alpha}}\Cp_{2k+1}(p)\big)|_{p=-1}&=\big(\partial_{p^{\alpha}}(e^{p})\big)|_{p=-1}=e^{-1}.
	\end{aligned}
\end{equation*}
Here $0\leq \alpha\leq k$ is an integer. 
Therefore, $g\in C^k(\bbR)$ and one has the estimate of $u_h^d$ as
\begin{equation*}
	\|u_0 - u_h^d\|_{L^2(\tilde \Omega_p)\times L^2(\Omega_x)}
	\lesssim (\triangle p^{k+1} + \triangle x^s)\|u_0\|_{H^s(\Omega_x)} +\varepsilon.
\end{equation*}
The choice of $k$ can be chosen such that $\triangle p^{k+1}  = \mathscr{O}(\triangle x^s)$.  We use $k=1$ in our numerical experiments in  section~\ref{sec:numerical tests}.

	The linear combination of Hamiltonian simulation (LCHS) method introduced in  \cite{ALL2023LCH} employs the continuous Fourier transform in $p$, and then applying numerical integration to complete the inverse of Fourier transform, choosing $p=0$ to recover the target variables $u$. Initially, it is a first-order method. Subsequent enhancements in \cite{ACL2023LCH2} improve the accuracy of the original LCHS algorithms by introducing new kernel functions with faster decay and truncating the Fourier transform integral to a finite interval $[-X,X]$, where  $ X = (\log(1/\varepsilon))^{1/\beta} $ ($ 0 < \beta < 1 $) truncates the continuous Fourier variable $\xi$.  
	We emphasize that the fundamental principle remains consistent with our approach to smooth extension. Indeed,  the Fourier transforms of the kernel functions are $e^{-p}$ for $p>0$, whereas they exhibit significant differences on the negative real axis. Therefore, the transformed kernel function can be interpreted as a smooth extension in the $p$ space.

\section{Application to the convection equations with purely imaginary wave speed}

In this section, we concentrate on the linear convection equation with an imaginary convection speed which is unstable thus hard to simulate numerically. This is a simple model for more complicated, physically interesting problems such as Maxwell's equation with negative index of refraction \cite{shelby2001experimental, parazzoli2003experimental} that has applications in meta materials. 
We consider the one-dimensional model with periodic boundary condition
\begin{equation}\label{eq:imag convection eq}
	\partial_t v = i \partial_x v \quad \;\text{in}\;\Omega_x, \quad v(x,0) = v_0,
\end{equation}
where $\Omega_x =[-\pi,\pi]$. This problem is ill-posed since, by  taking a Fourier transform on $x$, one gets
\[ \partial_t \hat{v}= -\eta \hat{v},\]
where $\hat{v}$ is the Fourier transform of $v$ in $x$. 
Assume the solution of \eqref{eq:imag convection eq} in Sobolev space exits, the exact solution is 
\begin{equation}
	v(x,t) = \int_{\bbR} e^{-ix\eta}e^{-\eta t}\hat v_0\;\d \eta.    
\end{equation}
Clearly, the solution contains exponential growing modes corresponding to $\eta<0$.  
Suppose $\hat v_0 \in H_0^s$, then $v(t,x)\in H^s(\Omega_x)$. 
The Schr\"odingerisation gives 
\begin{equation}\label{eq:imag convection eq w}
	\begin{aligned}
		\frac{\D}{\D t} w = -i\partial_{xp} w\quad \text{in}\;\Omega_x\times \Omega_p, \quad 
		w(0,x,p) = e^{-|p|}v_0(x) \quad \text{in}\;\Omega_x\times \Omega_p,
	\end{aligned}
\end{equation}
where $\Omega_p = (-\infty,\infty)$. 
The Fourier transform shows
\begin{equation}\label{eq:imag hat w}
	\frac{\D}{\D t} \hat w =  i\xi \eta \hat w, \quad \hat{w}(0,\xi,\eta)=\frac{\hat v_0(\eta)}{\pi(1+\xi^2)},
\end{equation}
which has a bounded solution for all time, with the exact solution given by
$$w(t,x,p) =\int_{\bbR}\int_{\bbR} e^{-i (x\eta + p\xi)} e^{i\xi\eta t}\hat w_0\;\d\eta \d \xi=\int_{\bbR} e^{-|p-\eta t|} \hat{v}_0(\eta) e^{-ix\eta}\;\d \eta. $$
Obviously, the $w$-equation \eqref{eq:imag convection eq w} and the Hamiltonian system \eqref{eq:imag hat w} can be simulated by a stable scheme in both quantum computers and classical ones. 
We define an approximate solution to the truncation as follows.
\begin{definition}
	Define 
	\begin{equation}
		w_{\eta_{\max}} = \int_{\bbR} e^{-|p-\eta t |}\hat v_0 \chi_{\max} e^{-ix\eta}\;\d \eta,
	\end{equation}
	where $\chi_{\max}$ is a characteristic function of $[-\eta_{\max},\eta_{\max}]$. The truncated solution is 
	\begin{equation}\label{eq:v_etamax}
		v_{\eta_{\max}}(t,x,p) = e^p w_{\eta_{\max}}(t,x,p),\quad 
		\text{or}\;\quad 
		v_{\eta_{\max}}(t,x,p) = e^p \int_{p}^{\infty} w(t,x,q)\;\d q, \quad 
		p>p^{\Diamond},
	\end{equation}
	with $p^{\Diamond} = \eta_{\max} T$.
\end{definition}

Following Theorem~\ref{thm:recovery of u}, one gets the estimate for recovery. The proof is omitted here.
\begin{theorem}\label{thm:recovery of u imag conv}
	Let $v(t,x)$ and $ w(t,x,p)$  be the exact solution of \eqref{eq:imag convection eq}, \eqref{eq:imag convection eq w}, respectively,
	on the interval $t\in[0,T]$. Suppose there exists an a priori bound $\|v_T\|_{H^s(\Omega_x)}$ and $v_{\eta_{\max},T} = v_{\eta_{\max}}(T,x,p)$ is given by \eqref{eq:v_etamax} where $p>p^{\Diamond} = \eta_{\max} T$. Then 
	\begin{equation}
		\|v_T-v_{\eta_{\max},T}\|_{L^2(\Omega_x)} \leq 
		\frac{\delta(\eta_{\max})}{\eta_{\max}^s},
	\end{equation} 
	where $\delta(\eta_{\max}) = (\int_{\eta_{\max}}^{\infty} (1+\eta^2)^s |\hat v_T|^2\;\d\eta)^{\frac 12}$. 
\end{theorem}

Similar to Remark~\ref{remark:recovery u}, 
it follows from $v_T\in H^s(\Omega_x)$ that $\delta(\eta_{\max})$ is a small number. In order to keep the error within $\varepsilon$, one could choose 
$\eta_{\max} \geq (\frac{\delta}{\varepsilon})^{\frac{1}{s}}$. For example, 
if $u = \alpha(t)\cos(\omega_0 x )$,
then one gets $v_T = e^{p}w(T,x,p)$ for $p>p^{\Diamond} = \eta_{\max}T = \omega_0 T$.

\subsection{Discretization for Schr\"odingerisation of the convection equations with imaginary convective terms}

In this section, we give the discretization of Schr\"odingerisation for the convection equations with  purely  imaginary wave speed. 
If we use the spectral method, the algorithm is quite similar to section~\ref{sec:quantum algorithms for the backward Eq}.
It is obtained by letting $A = D_x$ in \eqref{eq:tilde w}. 
It is obvious to see that $A$ is Hermitian and has both positive and negative eigenvalues. 
It is hard to construct a stable scheme if one solves the original problem.   
However, after Schr\"odingerisation, it becomes a Hamiltonian system which is quite easy to 
approximate in both classical and  quantum computers. Applying Theorem~\ref{thm:recovery of u imag conv}, one gets the desired variables.

\section{Numerical tests}\label{sec:numerical tests}
In this section, we perform several numerical simulations for the backward heat equation and imaginary convection equation by Schr\"odingeriza-tion in order to check the performance 
(in recovery  and convergence rates) of our methods. 
For all numerical simulations, we perform the tests in the classical computers by using the Crank-Nicolson method for temporal discretization. 
In addition, we apply the trapezoidal rule for numerical integration of all of the tests. In order to get higher-order convergence rates of discretization~\eqref{eq:tilde w}, we smooth the initial data $w(T)$ in \eqref{w-heat} by choosing
\begin{equation}
	g(p) =  \begin{cases}
		(-3+3e^{-1})p^3+(-5+4e^{-1})p^2-p+1  \quad p\in(-1,0), \\
		e^{-|p|} \quad p\in (-\infty,-1]\cup[0,\infty).
	\end{cases}
\end{equation}
Therefore, $g(p)\in H^2(\Omega_p)$.

\subsection{Approximation of smooth solutions}

These tests are used to evaluate the recovery  and the order of  accuracy of our method. More precisely, we want to show the  choice of $p^{\Diamond}$ in \eqref{eq:choice of pDiamond} and convergence rates in Theorems~\ref{thm:err of uhd-u}.

In the first test, we consider the backward heat equation in $\Omega_x=[0,2]$ with Dirichlet boundary condition $u(T,0) = u(T,2) = 0$ and $T=1$. We take a smooth solution 
\begin{equation}\label{eq:test1}
	u = \exp(-\frac{\pi^2}{4}t)\sin(\frac{\pi x}{2}),
\end{equation}
and the initial data is $u_T = \exp(-\frac{\pi^2}{4})\sin(\frac{\pi x}{2})$. We use the finite difference method for  spatial discretization with $A$ defined by
\begin{equation}\label{eq:spatial discretization}
	A =\frac{1}{\triangle x^2} \begin{bmatrix}
		-2 &1 & & & &\\
		1 &-2 & 1 & & &\\
		&\ddots  &\ddots  &\ddots  &\\
		& &1  & -2 & 1\\
		& & & 1 & -2
	\end{bmatrix},
\end{equation}
where $\triangle x$ is the mesh size.
The computation stops at $t=0$. According to \eqref{eq:choice of pDiamond}, one has $\eta_{\max} = \frac{\pi}{2}$ and $p^{\Diamond} = \frac{\pi^2}{4}$. The numerical results are shown in Fig.~\ref{fig:recovery of u} and Tab.~\ref{tab:err of recovery dis conti}. 
From  the plot on the left in Fig.~\ref{fig:recovery of u}, it can be seen that the error between $u$ and $u_h^d = w_h^d e^{p}$ drops precipitously at $p^{\Diamond}$, and  numerical solutions of Schr\"odingerisation recovered by choosing one point or numerical integration are close to the  exact solution. 
According to Tab.~\ref{tab:err of recovery dis conti}, second-order convergence rates of $\|w_h^de^p-u\|_{L^2(\Omega_x)\times L^2(\tilde \Omega_p)}$  and $\|u_d^{*}-u\|_{L^2(\Omega_x)}$ are obtained, respectively, where $u_d^* = \frac{\int_{3}^{10} w_h^d\;\d p}{e^{-3}-e^{-10}}$.

\begin{figure}[http]
	\quad \quad 
	\includegraphics[width=0.4\linewidth]{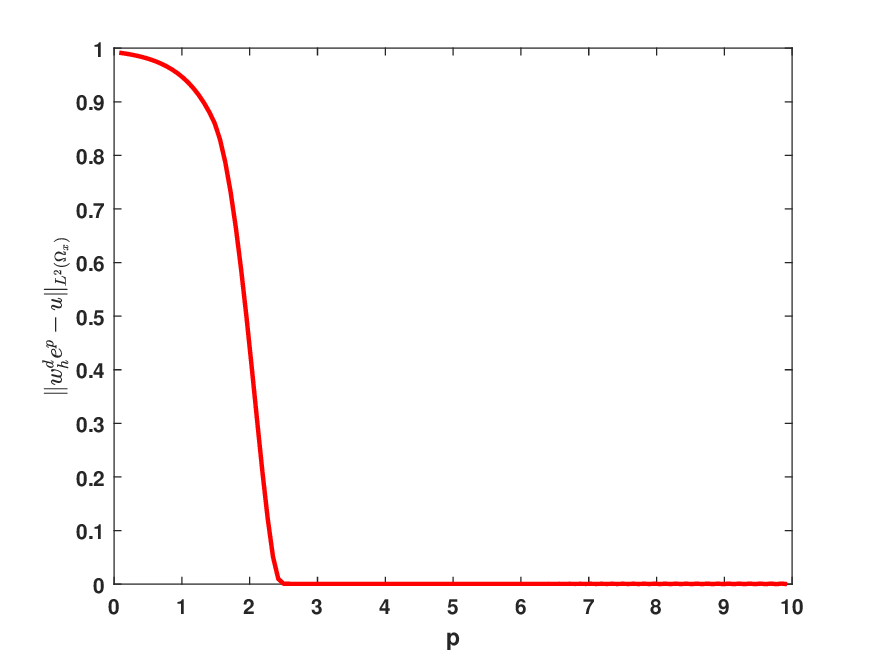}
	\includegraphics[width=0.4\linewidth]{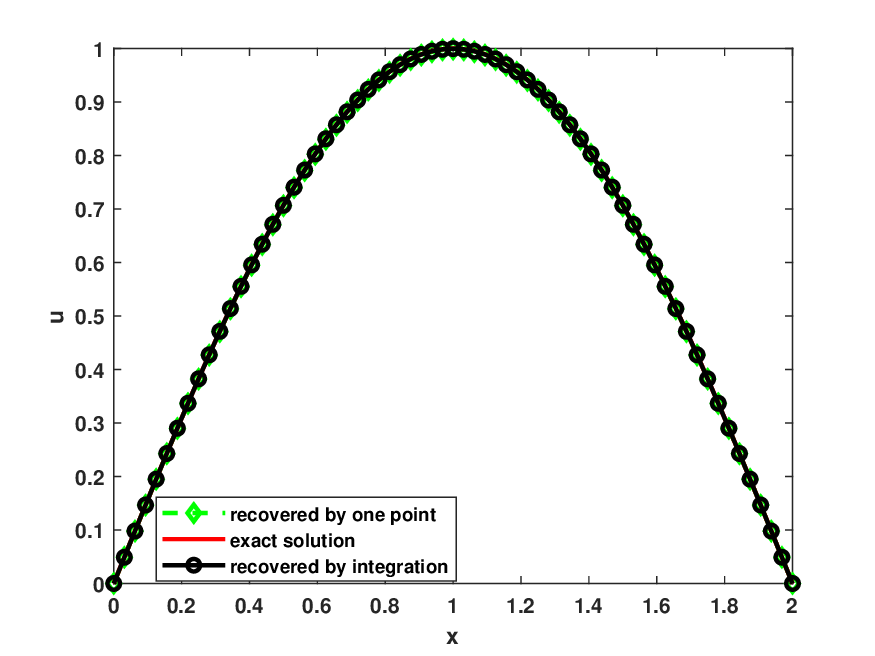}
	\caption{ Left: $\|w_h^de^{p}-u\|_{L^2(\Omega_x)}$
		with respect to $p$, with $w_h^d$ computed by \eqref{eq:tilde w}-\eqref{eq:whd(t,x,p)}. Right: the recovery from Schr\"odingerisation by choosing $p>p^{\Diamond}= \pi^2/4 $.	
	}\label{fig:recovery of u}
\end{figure}
\begin{table}[htbp]
	\centering
	\begin{tabular}{lcccccc}
		\midrule [2pt]
		$(\triangle p, \triangle x)$          & $(\frac{10}{2^7}, \frac{1}{2^5})$ & order & $(\frac{10}{2^8}, \frac{1}{2^6})$ & order & $(\frac{10}{2^9}, \frac{1}{2^7})$ & order\\ 
		\hline
		$\|u_h^d-u\|_{L^2(\Omega_x)\times L^2(\tilde \Omega_p)}$       
		&1.64e-03  &-  &3.25e-04 &2.33  &8.05e-05 &2.01 \\ \hline
		$\|u_d^*-u\|_{L^2(\Omega_x)}$ 
		&7.48e-04  &- &1.86e-04  &2.01 &4.56e-05 &2.00\\
		
		\midrule [2pt]
	\end{tabular}            
	\caption{
		The convergence rates of  $\|u_h^d-u\|_{L^2(\Omega_x)\times L^2(\tilde \Omega_p)}$ and $\|u_d^*-u\|_{L^2(\Omega_x)}$, where $u_h^d$ is computed by  \eqref{eq:tilde w}-\eqref{eq:uhd} and $u_d^* = \frac{\int_{3}^{10} w_h^d\;\d p}{e^{-3}-e^{-10}}$, $\tilde \Omega_p = [3,10]$ and $\triangle t = \frac{\triangle x}{2^3}$.
	}\label{tab:err of recovery dis conti}
\end{table}

Next, we consider another exact solution set by
\begin{equation}
	u = \exp(- 9\pi^2 t) \sin(3\pi x) \quad \text{in}\; [0,2],
\end{equation}
with periodic boundary conditions and the initial data is $u_T = \exp(-9\pi^2) \sin (3\pi x)$.  
We apply the spectral method to discrete the $x$-domain.
From Theorem~\ref{thm:recovery of u} and Remark~\ref{remark:recovery u}, we deduce that $p^{\Diamond} = \omega_0^2 T = 9\pi^2\approx 90$.
From Fig.~\ref{fig:recovery of u2}, it can be seen that $u = e^{p} w$ when $p\geq p^{\Diamond}\approx 90$, and the numerical solutions are very close to the exact solution.
In Tab.~\ref{tab:err of recovery conti}, the $L^2$ errors of the method \eqref{eq:tilde w}-\eqref{eq:uhd} are presented. They show the optimal order 
$\|u_h^d-u\|_{L^2(\Omega_x)\times L^2(\tilde \Omega_p)}\sim \triangle p^2$, $\|u_d^*-u\|_{L^2(\Omega_x)}\sim \triangle p^2$ are obtained  when $\triangle p\to 0$, where 
$\tilde \Omega_p =[90,95]$, and $u_h^d = e^p w_h^d$, $u_d^* =  \frac{\int_{90}^{95}w_h^d\;\d p}{e^{-90}-e^{-95}}$.

\begin{figure}[t!]
	\quad \quad 
	\includegraphics[width=0.42\linewidth]{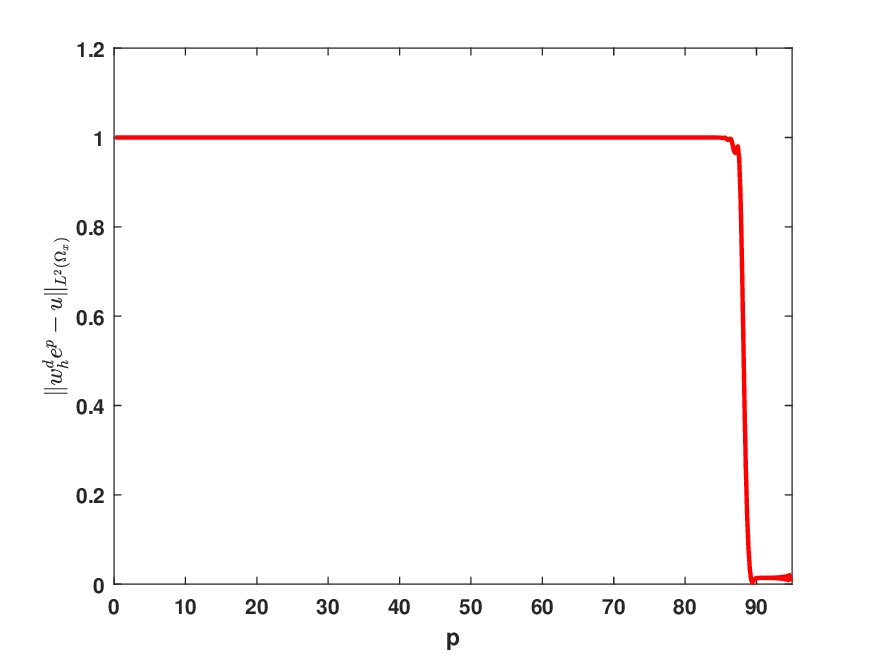} 
	\includegraphics[width=0.42\linewidth]{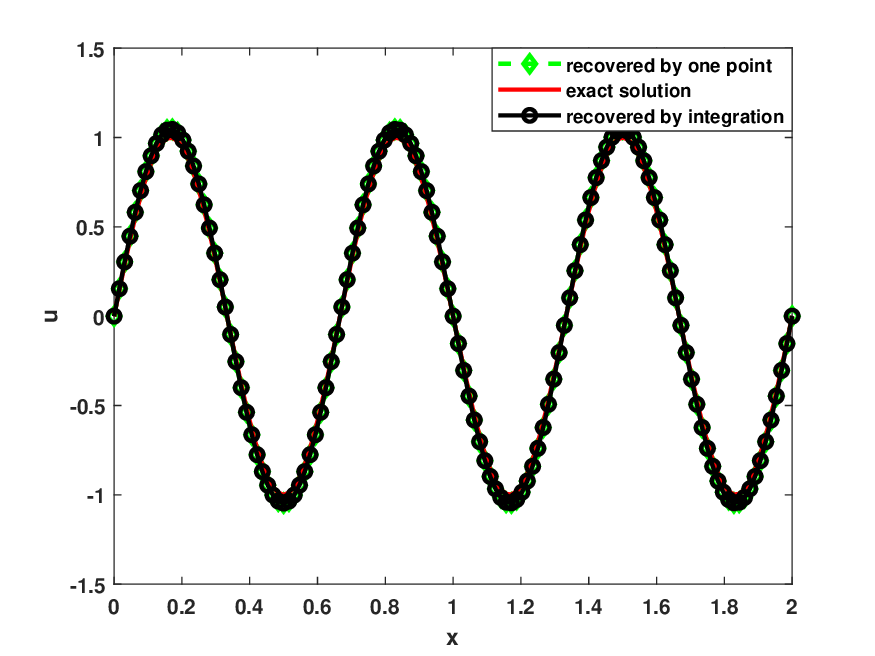}
	\caption{ Left: $\|w_h^de^{p}-u\|_{L^2(\Omega_x)}$
		with respect to $p$, with $w_h^d$ computed by \eqref{eq:tilde w}-\eqref{eq:whd(t,x,p)}. Right: the recovery from Schr\"odingerisation by choosing $p>p^{\Diamond} = 9\pi^2$.	
	}\label{fig:recovery of u2}
\end{figure}
\begin{table}[htbp]
	\centering
	\begin{tabular}{lcccccc}
		\midrule [2pt]
		$(\triangle p,\triangle x)$          & $(\frac{95}{2^8},\frac{1}{2^6})$ & order & $(\frac{95}{2^9},\frac{1}{2^7})$ & order & $(\frac{95}{2^{10}},\frac{1}{2^8})$ & order\\ 
		\hline
		$\|u_h^d-u\|_{L^2(\Omega_x)\times L^2(\tilde \Omega_p)}$       
		&4.73e-01   &- &1.21e-01 &1.96  &3.07e-02 &1.97 \\ \hline
		$\|u_d^*-u\|_{L^2(\Omega_x)}$ 
		&1.37e-01  &- &5.08e-02 &1.43  &1.35e-02 &1.90\\
		\midrule [2pt]
	\end{tabular}            
	\caption{
		The convergence rates of  $\|u_h^d-u\|_{L^2(\Omega_x)\times L^2(\tilde \Omega_p)}$ and $\|u_d^*-u\|_{L^2(\Omega_x)}$, respectively,  where $u_h^d$ is computed by  \eqref{eq:whd(t,x,p)}-\eqref{eq:uhd} and $u_d^* = \frac{\int_{90}^{95}w_h^d\;\d p}{e^{-90}-e^{-95}}$, $\tilde \Omega_p =[90,95]$, $\triangle t = \frac{\triangle x}{2^4}$.
	}\label{tab:err of recovery conti}
\end{table}

\subsection{Approximation of a piece-wise smooth solution}\label{sec:Approximation of a piece-wise smooth solution}
Now, we consider a piece-wise smooth solution in $(-\infty,\infty)$
\begin{equation}\label{eq:initial heat Eq}
	u(x,0) =  \begin{cases}
		-|x|+1 \quad x\in [-1,1],\\
		0 \quad otherwise.
	\end{cases}
\end{equation} 
The tests are done in the interval $x\in [-20,20]$ and  $\triangle x = \frac{40}{2^7}$, $\triangle t = \frac{1}{2^8}$. We use the finite difference method to simulate the heat equation with the initial condition given by \eqref{eq:initial heat Eq}. Thus, we get the approximate solution to $u_T$ at $T=1$.  According to Theorem~\ref{thm:recovery of u},
we choose $\varepsilon = \triangle x$, $\delta(\eta_{\max}) = 1$ and $s=2$, $2$ and $1$ for recovering the target variables $u(t)|_{t=0.5}$, $u(t)|_{t=0.25}$ and $u(t)|_{t=0}$, respectively.
Thus, it yields 
$\eta_{\max} = (\frac{1}{\triangle x})^{1/2}= \sqrt{3.2}$, $(\frac{1}{\triangle x})^{1/2}= \sqrt{3.2}$ and
$\frac{1}{\triangle x}=3.2$.
The results are shown in Fig.~\ref{fig:recovery of u_3} with $u_h^d$ computed by 
\eqref{eq:tilde w}-\eqref{eq:uhd} and $p\in [-30,30]$, $\triangle p = \frac{60}{2^{10}}$. From the plot of Fig.~\ref{fig:recovery of u_3}, 
we find that the error is larger when $t$ is smaller, which is consistent with the estimate in Theorem~\ref{thm:recovery of u}.
\begin{figure}[htbp]
	\includegraphics[width=0.31\linewidth]{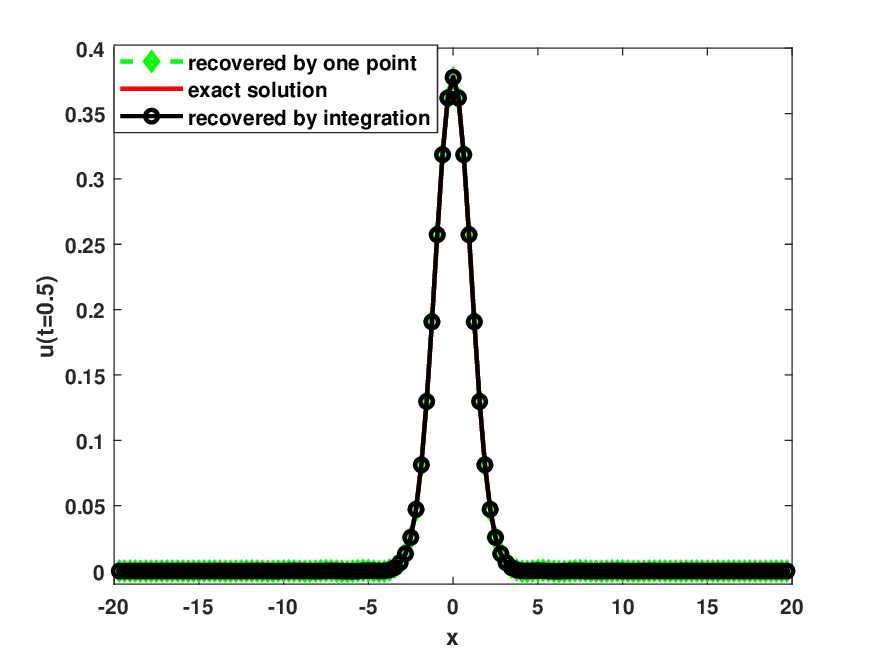}
	\includegraphics[width=0.31\linewidth]{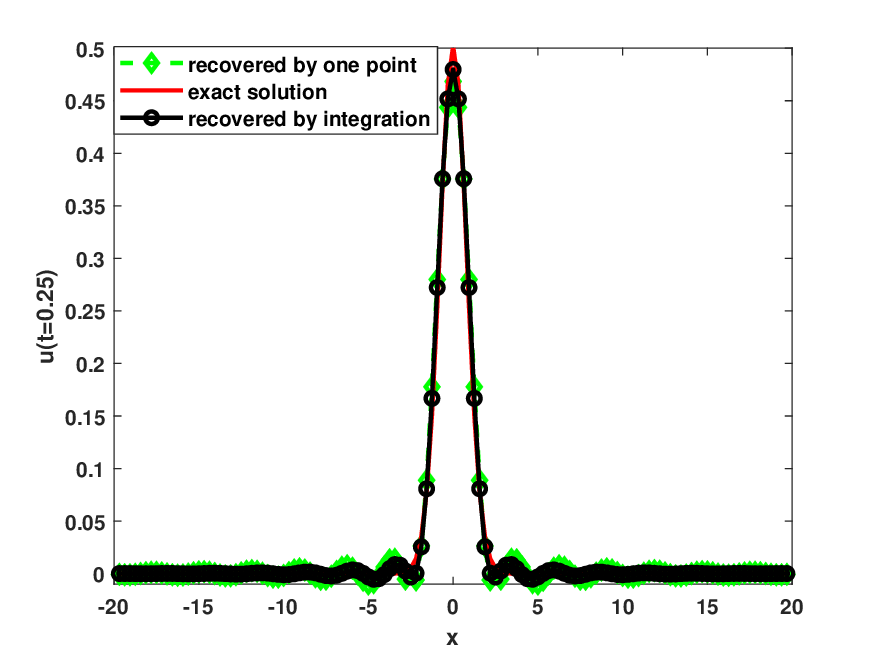}
	\includegraphics[width=0.31\linewidth]{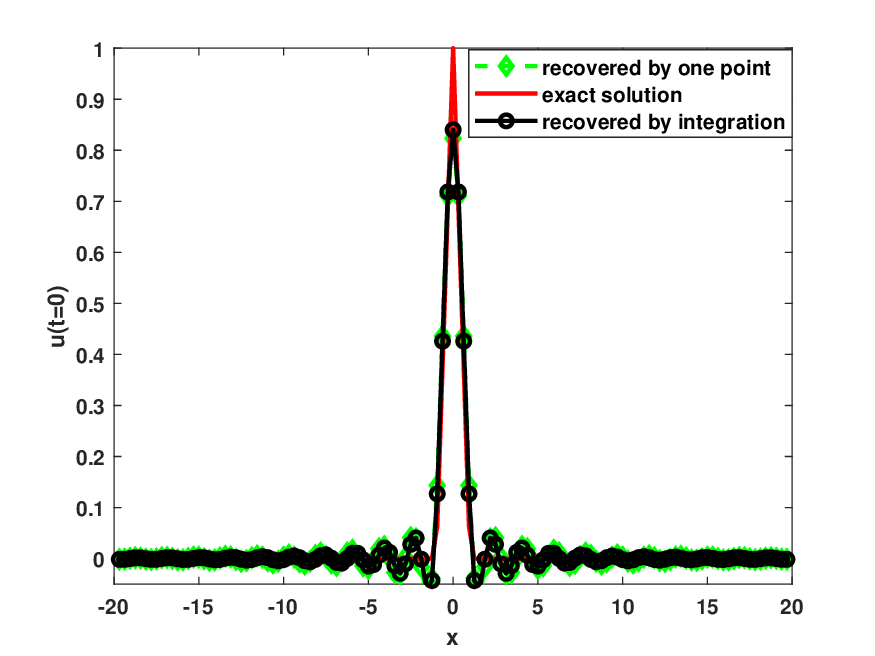}
	\caption{
		The exact solution in the legend is the numerical solution at time 
		$t=0.5$, $t=0.25$, and $t=0$, respectively, obtained by solving the forward heat conduction equation using the finite difference method.
		The results of Schr\"odingerisation are shown with  $\eta_{\max} = (\frac{1}{\triangle x})^{1/2} = \sqrt{3.2},(\frac{1}{\triangle x})^{1/2} = \sqrt{3.2}, \frac{1}{\triangle x} =3.2$, respectively.
	}\label{fig:recovery of u_3}
\end{figure}

\subsection{Input data with noises}

In this test, we investigate the cases with the input data containing noises. The exact solution  in $\bbR$ is set by 
\begin{equation}
	u(x,t) = \frac{1}{\sqrt{1+4t}}e^{-\frac{x^2}{1+4t}}.
\end{equation} 
The input data is $u_T^{\zeta} = u_T + \zeta_0 \, \text{rand}(x)$. The magnitude $\zeta_0$ indicates the noise level of the measurement data, and $\text{rand}(x)$ is a random number such that $-1\leq \text{rand}(x)\leq 1 $.
The tests are truncated to the interval $ [-10,10]$ and $T=1$. The noise levels are $\zeta_0 = 6\times 10^{-2}$, $6\times 10^{-3}$ and $6\times 10^{-4}$. We apply \eqref{eq:tilde w}-\eqref{eq:uhd} to discretize the Schr\"odingerisation with $p \in [-20,20]$ and $\triangle p = \frac{10}{2^6}$.
We use the finite difference method to discretize the $x$-domain with $\triangle x = \frac{10}{2^6}$, $\triangle t = \frac{1}{2^{10}}$.
The results are shown in Fig.~\ref{fig:input with noise}.
According to Theorem~\ref{thm:recover of u with input noise}, it is hard to find $\eta_{\max}$ to satisfy $\frac{\delta(\eta_{\max})}{\eta_{\max}^s} + e^{\eta_{\max}^2 T}\zeta_0 \leq \varepsilon$ for any fixed $\zeta_0$, $T$ and $\varepsilon$.   
However one could choose 
$\eta_{\max} = \sqrt{ \ln \big( (\frac{1}{\zeta_0})^{\frac{1}{T}} (\ln \frac{1}{\zeta_0})^{\frac{-s}{2T}}\big)}$
to get $\|u(t)-u_{\eta_{\max},\zeta}(t)\|_{L^2(\Omega_x)}\lesssim (\ln(\frac{1}{\zeta_0}))^{-\frac{s}{2T}}\to 0$ as 
$\zeta_0$ tends to zero  under the  assumption $T\leq \ln (1/\zeta_0)$, $\zeta_0<1$.
By investigating the cases of different $\zeta_0$, we find the numerical solution of Schr\"odingerisation approximates $u_{\eta_{\max},\zeta}$ and tends to the exact solution as $\zeta_0$ goes to zero.
The oscillation of the numerical solution comes from two aspects, one is the truncation of the calculation area of $x$, and the other is the finite difference method of calculating the discontinuous input data. 

\begin{figure}[htbp]
	\includegraphics[width=0.3\linewidth]{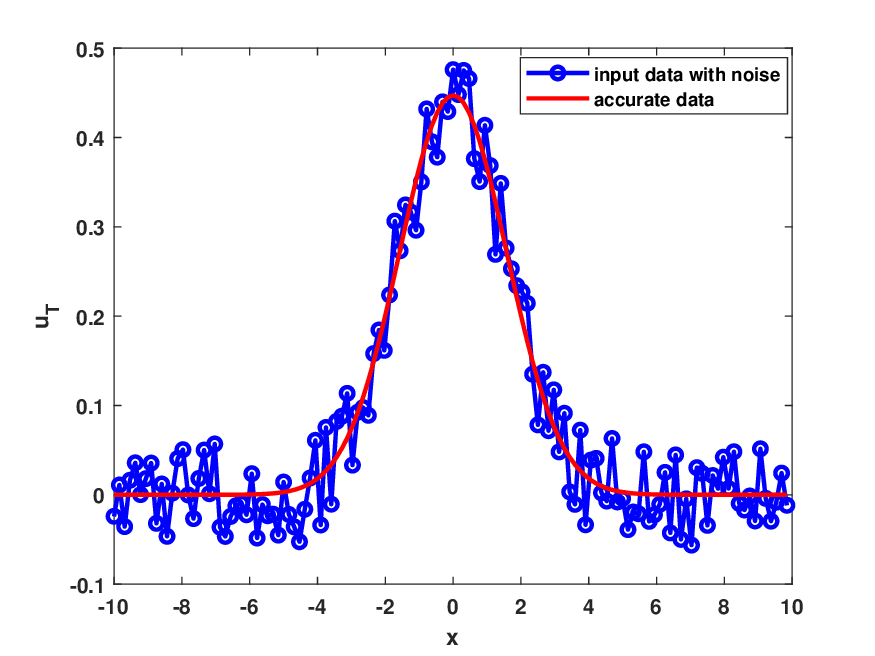}
	\includegraphics[width=0.3\linewidth]{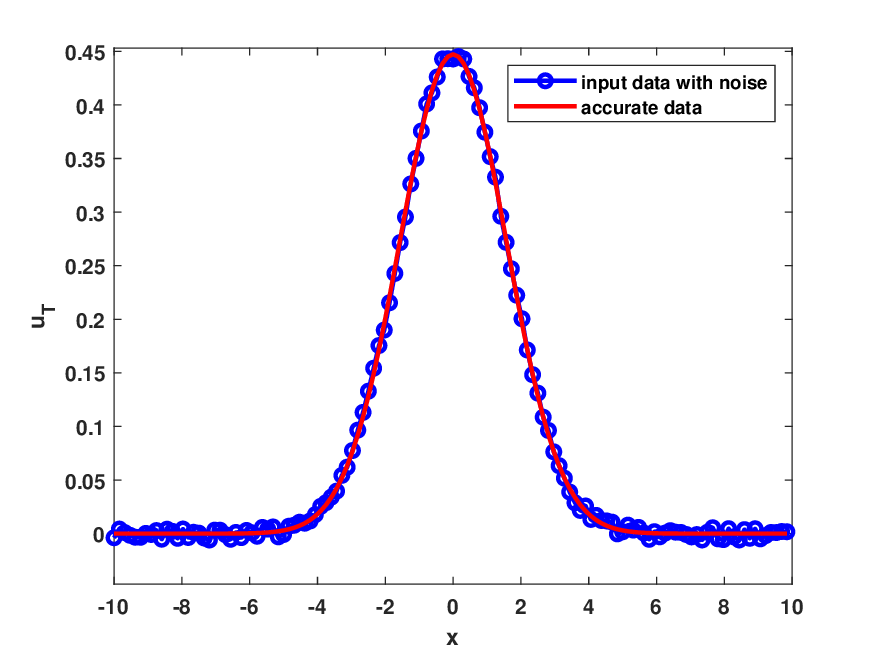}
	\includegraphics[width=0.3\linewidth]{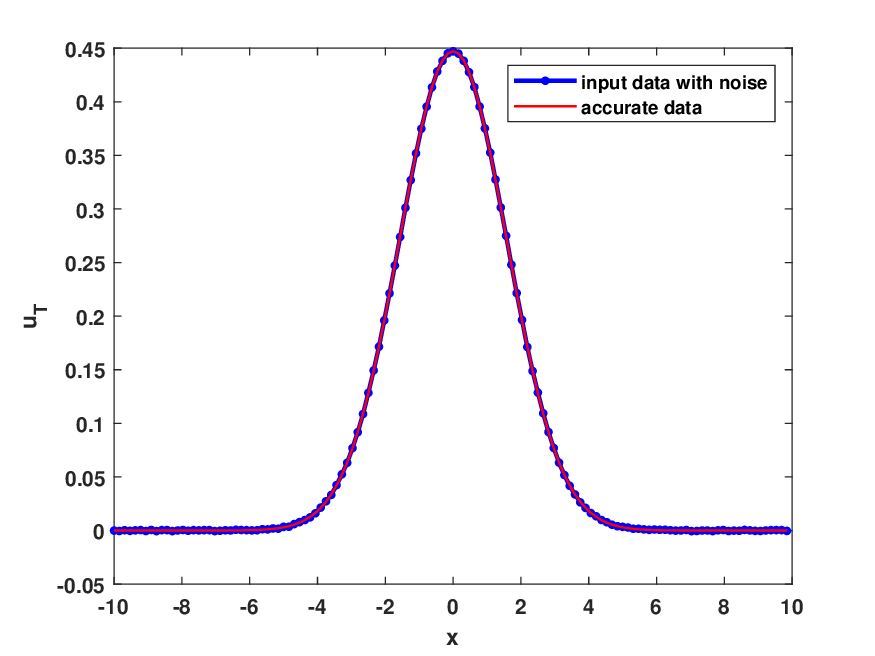}\\
	\includegraphics[width=0.3\linewidth]{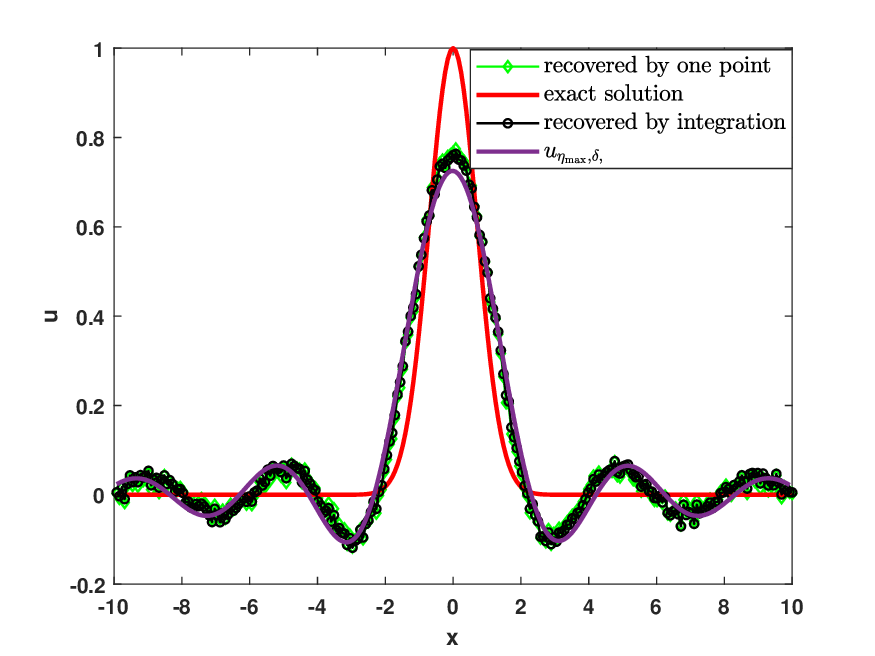}
	\includegraphics[width=0.3\linewidth]{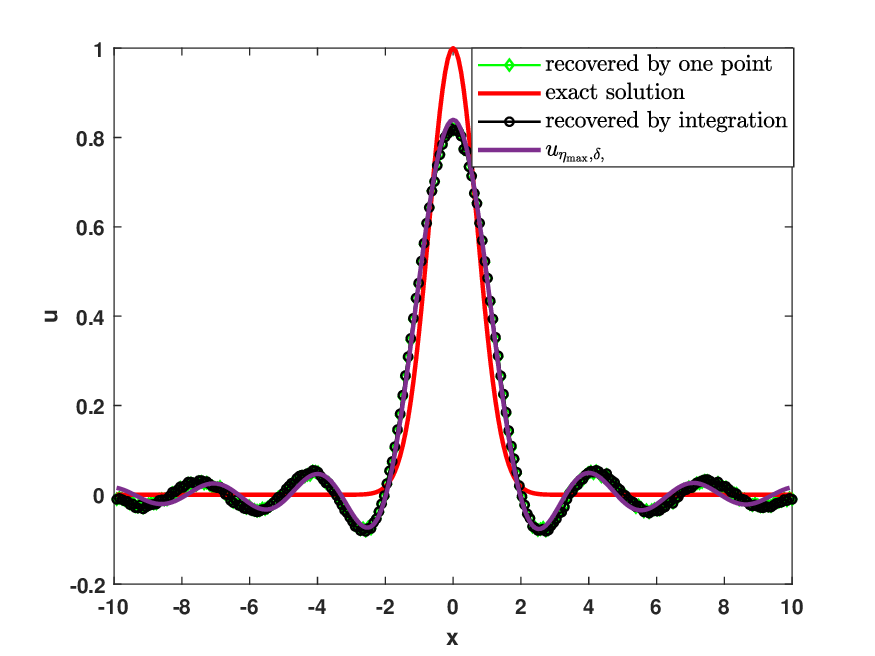}
	\includegraphics[width=0.3\linewidth]{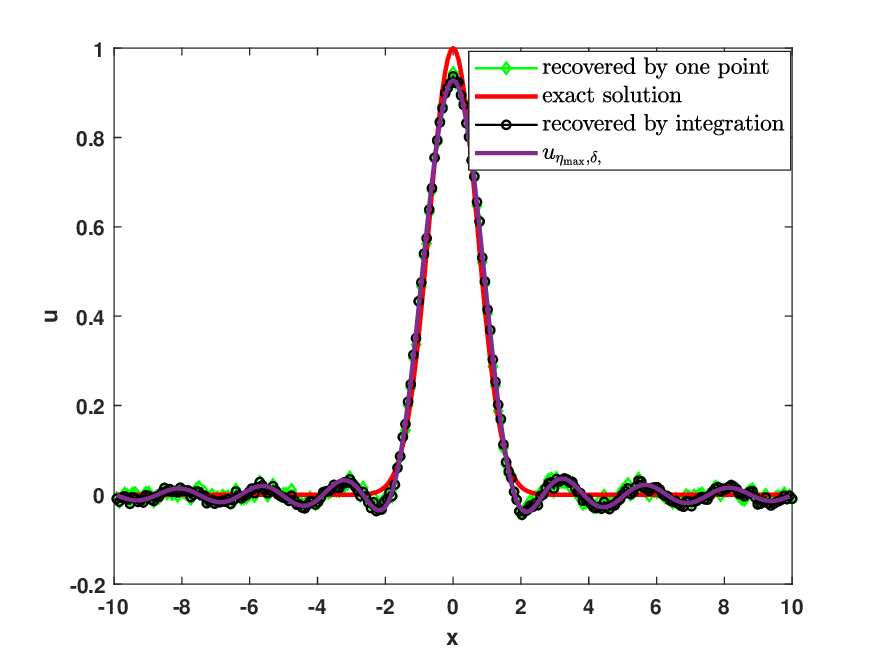}
	\caption{The first row: input data with different noise levels from left to right. The noise levels are 
		$6\times 10^{-2}$, $6\times 10^{-3}$ and $6\times 10^{-4}$.
		The second row: results of Schr\"odingerisation computed by \eqref{eq:tilde w}-\eqref{eq:uhd} using the input data above with $\eta_{\max} = 1.5,2,2.5$,respectively.
	}\label{fig:input with noise}
\end{figure}

\subsection{The  convection equation with imaginary wave speed}

Now we consider the  convection equations with imaginary wave speed to test the accuracy of the recovery. In this test, the exact solution is 
\begin{equation}\label{eq:test exactu}
	v(t,x) = \cos(3\pi(x+t)),
\end{equation}
and the simulation stops at $T=1$. Therefore the source term is obtained by $f= \partial_t u-i \partial_x u$ with $u$ defined in \eqref{eq:test exactu}. We use the spectral method on $p$ and $x$, then get the linear system:
\begin{equation}\label{eq:tilde wh imag}
	\frac{\D}{\D t} \tilde \Bw_h  = i(D_p\otimes D_x) \tilde \Bw_h + \tilde \Bf_h, \quad  
	\tilde\Bw_h(0) = (\Phi_p^{-1}\otimes\Phi^{-1}) \Bw_h(0),
\end{equation}
where $\tilde \Bf_h = (\Phi_p^{-1}\otimes\Phi^{-1})\Bf_h$, with $\Bf_h = \sum\limits_{j=0}^{M-1}\sum\limits_{l=0}^{N-1} f(x_j,t)e^{-|p_l|}(\Be_l^{(N)}\otimes \Be_j^{(M)})$.  
In order to obtain unitary dynamical systems to facilitate operation on quantum computers,  one needs to 
use the technique of dealing with source terms in \cite{JLM24}.
In Fig.~\ref{fig:image_convection}, one can find $p^{\Diamond} \approx 3\pi T $ which confirms Theorem~\ref{thm:recovery of u imag conv}, and the numerical solutions are close to the exact one. 
Since the solution is not smooth enough when using the spectral method, the oscillation appears in the error between $v$ and $v_h^d$ in Fig.~\ref{fig:image_convection}. 

\begin{figure}[htbp]
	\quad \quad 
	\includegraphics[width=0.42\linewidth]{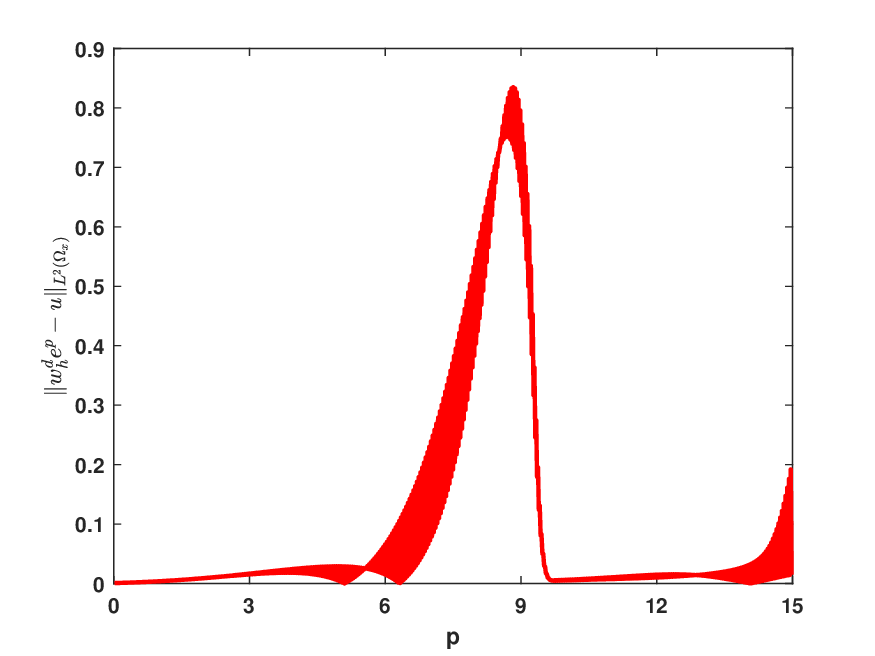}
	\includegraphics[width=0.42\linewidth]{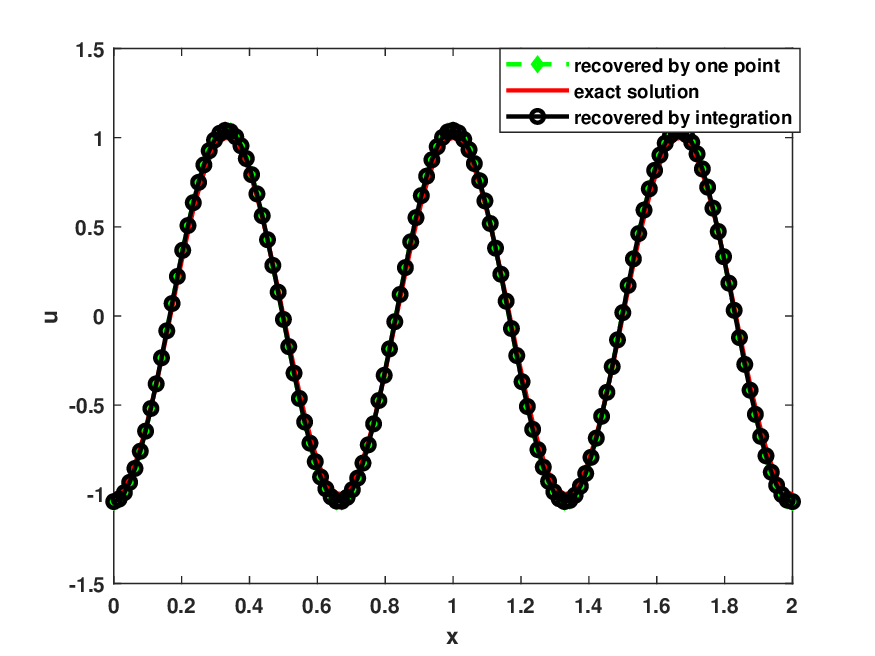}
	\caption{ Left: $\|w_h^de^{p}-v\|_{L^2(\Omega_x)}$
		with respect to $p$, where $w_h^d$ is from \eqref{eq:tilde w}. Right: the recovery from Schr\"odingerisation by choosing $p>p^{\Diamond}= 10$, where the numerical solution is obtained by  $v_d^*=\frac{\int_{10}^{15}w_h^d\;\d p}{e^{-10}-e^{-15}}$, 
		$p\in [-20,20]$, $\triangle p = \frac{10}{2^8}$ and $\triangle x = \frac{1}{2^6}$.
	}\label{fig:image_convection}
\end{figure}

\section{Quantum algorithms for ill-posed problems}\label{sec:quantum}
The above numerical sch- eme is a new scheme that can be used on both classical and quantum devices based on qubits or qumodes (continuous variables)\cite{CVPDE2023}.
\subsection{Qubit-based quantum computing}
	Since Eq.\eqref{eq:tilde w} is a Hamiltonian system,  a quantum simulation can be carried out such as quantum signal processing (QSP)\cite{QSP17}, linear combination of unitaries (LCU) \cite{LCU12}, etc.
	The complete circuit for implementing the quantum simulation of $\ket{\Bu_h(t)} = \Bu_h(t)/\|\Bu_h(t)\|$ is illustrated in Fig.~\ref{schr_circuit}, where $\ket{\Bu_h(t)}$ is a quantum state.
	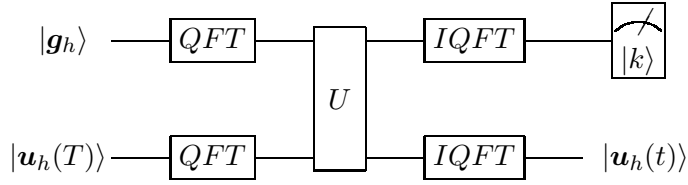
\begin{figure}[t!] 
		\centerline{
			\Qcircuit @C=1em @R=2em {
				\lstick{\hbox to  2em{$\ket{\Bg_h}$\hss}}
				& \qw
				& \gate{QFT}
				& \qw
				& \multigate{1}{U}
				& \qw
				& \gate{IQFT}
				& \qw
				& \qw  &\meterB{\ket{k}}\\
				\lstick{\hbox to  3em{$ \ket{\Bu_h(T)}$\hss}} 
				& \qw
				& \gate{QFT}
				& \qw	
				& \ghost{U}	
				& \qw 
				& \gate{IQFT}
				& \qw	 
				& \qw  & \hbox to 2.5em{$\ket{\Bu_h(t)}$\hss} 
			}  
		}
		\caption{Quantum circuit for Schr\"odingerisation of Eq.~\eqref{eq:tilde uh}, $U$ is a unitary matrix approximating $\Cu = \exp(i (D_p\otimes A) (T-t))$, and  QFT (IQFT) denotes the (inverse) quantum Fourier transform.}
		\label{schr_circuit}
	\end{figure}

For the specific ill-posed problems, the difference from well-posed problems lie in  (i) the cost of unitarily evolving the system in Eq.~\eqref{eq:tilde w}, and (ii) the cost of final measurement to retrieve $|\Bu_h(t)\rangle \equiv \vect{u}_h(t)/\|\Bu_h(t)\|$, $t<T$, depending on the specification of possible $p^{\Diamond}$ for different problems.

Since the evolution is unitary, the cost analysis is similar whether one is running forward or backward in time. The only difference is that the initial state used in the forward equation is $|\Bu_h(0)\rangle$ plus the ancilla state, whereas in the backward equation it is $|\Bu_h(T)\rangle$ plus the same ancilla state (see Fig. ~\ref{schr_circuit}).   
	In quantum algorithms, Hamiltonian simulation with nearly optimal dependence on all parameters can be found in \cite{BCK15}, with complexity given by the next lemma.
	\begin{lemma}\label{lem:hamiltonian complexity}
		For a Hamiltonian system $\frac{\D}{\D t} \Bu = -i H \Bu$ with $H$ a  Hermitian matrix acting on $m_H$ qubits can be simulated within error $\varepsilon$ with $\mathscr{O}\big( \tau \log(\frac{\tau}{\varepsilon})/\log\log (\frac{\tau}{\varepsilon})\big)$
		queries and  
		$\mathscr{O}\big(
		\tau  m_H L_{\text{polylog}}\big)$
		additional 2-qubits gates, where $\tau = s(H) \|H\|_{\max}T$, $s(H)$ is the sparsity of $H$ (maximum number of nonzero entries in each row),  $\|H\|_{\max}$ is the max-norm
		(value of largest entry in absolute value),
		$T$ is the evolution time, and 
		\begin{equation*}
			L_{\text{polylog}} = \big [1+\log_{2.5}(\tau /\varepsilon)\big]
			\frac{\log(\tau /\varepsilon)}{\log\log(\tau /\varepsilon)}.
		\end{equation*}
	\end{lemma} 
	
	In the case of Eq.\eqref{eq:tilde w}, since the spatial discretization is a second-order scheme $\triangle x^2\sim \varepsilon$, we need $\mathscr{O}(d\log(\frac{1}{\sqrt{\varepsilon}}))$ qubits  for $x$ variable where $d$ is the dimension. 
	It is crucial to note that if one employs the smooth initialization method for the Schr\"odingerized equation, as introduced in section \ref{higher-p}, and discretizes this smooth function in terms of $p$ instead, both the initial data and the solution in $p$ will belong to the class $C^k$ for any integer $k\geq 1$. Consequently, if a spectral method is utilized, this approach will yield essentially spectral accuracy in the approximations within the extended $p$ space.
	The extra qubits used in the extension domain $p$ is almost $\mathscr{O} (\log\log(\frac{1}{\varepsilon}))$.
	Therefore, the required number of qubits is $$m_H = \mathscr{O}(d\log(\frac{1}{\sqrt{\varepsilon}}) + \log\log(\frac{1}{\varepsilon})),$$
	the Hermitian matrix
	$H = D_p\otimes A$ satisfies $\|H\|_{\max} = \mathscr{O}(\frac{1}{\varepsilon}\log(\frac{1}{\varepsilon}))$ and the sparsity $s(H)=\mathscr{O}(d)$.
	By applying Lemma~\ref{lem:hamiltonian complexity}, one can directly obtain the complexity of Schr\"odingerisation as follows.
	\begin{lemma}
		Given sparse-access to the Hermitian matrix $H=D_p\otimes A$,  and the  initial quantum state $\ket{\Bu_h(T)}$ has been obtained.  The quantum state $\ket{\Bu_h(t)}$ can be
		prepared  to precision $\varepsilon$ with $\tilde{\mathscr{O}}(\frac{dT}{\varepsilon})$ queries   and  $ \tilde{\mathscr{O}}(\frac{d^2T}{\varepsilon}) $
		additional 2-qubits gates, where $\tilde{\mathscr{O}}$ denotes the case where all logarithmic factors are suppressed.
	\end{lemma}
	
	This {\it reduces} the complexity of the Schr\"odingerisation based quantum algorithms for any dynamical system introduced in \cite{JLY22a, JLY22b}, wherein the number of 2-qubit gates required for the heat equation is $\tilde{\mathscr{O}}
	(\frac{d^2T}{\varepsilon^{2}})$.
	Furthermore, this quantum approach demonstrates polynomial-level acceleration compared to classical algorithms for solving the forward heat equation.

Different values of $p^{\Diamond}$ affect the probability of retrieving the correct final quantum state $|\Bu_h(t)\rangle$ for $t<T$. The final step involves measurement of $p$. 
	There are in principle two schemes from Theorem~\ref{thm:recovery u}, where both schemes give a success probability proportional to $K(\|u(t) \|^2_{L^2(\Omega_x)}/ \|u(T)\|^2_{L^2(\Omega_x)})$, thus 
	$\mathscr{O}(K^{-1}\|u(T)\|^2_{L^2(\Omega_x)}/\|u(t)\|^2_{L^2(\Omega_x)})$ measurements are needed, 
	where $K$ is a factor depending on $p^{\Diamond}= \eta_{\max}^2 T$, for instance see \cite{JLY22a,CVPDE2023}.
The first scheme is the recovery by measurement of one point in Eq. \eqref{eq:recover by one point}. This corresponds to accepting the resulting state so long as the measured $p>\eta_{\max}^2 T$. In the qubit context for example, the largest value 
	$K \sim e^{-2 \eta^2_{\max}T} \sim e^{-2(\frac{\delta}{\varepsilon})^{\frac{2}{s}} T}$, if $u_T\in H^s(\Omega_x)$. 
The second scheme is the recovery through the integration method in Eq.~\eqref{eq:recover by quad}. For a bounded domain, the retrieval process then involves the projection of $|w(T)\rangle$ onto the $[\eta^2_{\max} T, \pi L]$ domain.  In this case, $K \sim (\int_{p^{\Diamond}}^{\pi L} e^{-p} dp)^2=(e^{-\pi L}+e^{-p^{\Diamond}})^2 \leq (e^{-\pi L}+e^{-\eta^2_{\max}T})^2\sim (e^{-\pi L}+e^{- (\frac{\delta}{\varepsilon})^{\frac{2}{s}}})^2)$. Thus, for large domains where $\exp(-\pi L) \sim 0$, the first and second schemes are comparable. However, for specific cases, such as $u=\alpha(t)\cos(\omega_0 x)$ in Remark~\ref{remark:recovery u}, the probability to retrieving the state $\ket{\Bu_h(0)}$ is about 
	\begin{equation*}
		\frac{e^{-2\eta_{\max}^2 T} \|u(0)\|_{L^2(\Omega_x)}^2}{\|u(T)\|_{L^2(\Omega_x)}^2} = \frac{e^{-2\eta_{\max}^2T} \alpha(0)^2}{\alpha(T)^2} = \frac{e^{-2\eta_{\max}^2T} e^{2\omega_0^2T}\alpha(T)^2}{\alpha(T)^2} =e^{2(\omega_0^2-\eta_{\max}^2) T},
	\end{equation*}
	which is close to $1$ if we choose $\eta_{\max} = \omega_0$.  
	This condition aligns with the requirement for choosing $\eta_{\max}$ to successfully recover $u$, as analyzed in Remark \ref{remark:recovery u}.

\subsection{Qumode-based analog quantum computing}

Since Eq.~\eqref{eq:up} is already of Schr\"odinger's form and can be rewritten with 
$|w(t)\rangle=\int \vect{w}(t, p)|p\rangle dp/ \|\vect{w}(t,p)\|$ obeying 
\begin{align} \label{eq:schrform1}
	\frac{\partial |w(t)\rangle}{\partial t}=-i(H \otimes \hat{P})|w(t)\rangle, \qquad |w_0\rangle=\int e^{-|p|}|p\rangle dp |u_0\rangle, 
\end{align}
in the continuous-variable formalism and $\hat{P} \leftrightarrow -i \partial/\partial p$ is the momentum operator conjugate to the position operator $\hat{X}_p$ which has $p$ as eigenvalues and $[\hat{X}_p, \hat{P}]=iI$. Since $H \otimes \hat{P}$ is clearly Hermitian then $|w(0)\rangle$ evolves by unitary evolution generated by the Hamiltonian $H \otimes \hat{P}$. We note that here, since we actually want to run the original PDE backwards in time, the $|u_0\rangle$ state above corresponds in fact to the state $|u(t=T)\rangle$. If we want the state at some $t<T$, then the time $\mathcal{T}$ to run Eq.~\eqref{eq:schrform1} is over $\mathcal{T}\in [0, T-t]$.

In the continuous-variable formulation, how $\eta_{\max}<\infty$ arises is more subtle. The finite $\eta_{\max}$ can come instead from a kind of energy constraint $E_{\max}  \sim \eta^2_{\max}$ on the continuous-variable ancillary quantum mode, where $\eta_{\max}$ is acting like the maximum momentum for the mode (conjugate momentum to $p$). At the measurement step, here $p^{\Diamond}_{\max}=\pi L$ is actually a maximum value of measurable $p$ by the physical measurement apparatus. The detailed analysis here is more subtle and depends on the physical system studied, so we leave this to future work. 

\section{Conclusions}
In this paper we present computationally stable numerical methods for two ill-posed problems:  the (constant and variable coefficient) backward heat equation and the linear convection equation with  imaginary wave speed.  The main idea is to use Schr\"odingerisation, which turns  typical ill-posed  problems into unitary dynamical systems in one higher dimension.
The Schr\"odingerised system is a Hamiltonian system valid both forward and backward in time, hence suitable for the design of a computationally stable numerical scheme. A key idea is given to  recover the solution to the original problem from the Schr\"odingerised system using data from appropriate domain in the extended space.  Some sharp error estimates between the approximate solution and exact solution are provided and also verified numerically. Moreover, a smooth initialization is introduced for the Schr\"odingerised system which leads to requiring an exponentially lower number of ancilla qubits for qubit-based quantum algorithms for any non-unitary dynamical system. 

In the future this technique will be generalized to more realistic and physically interesting ill-posed problems. 
\section*{Acknowledgement}
SJ and NL  are supported by NSFC grant No. 12341104,  the Shanghai Jiao Tong University 2030 Initiative and the Fundamental Research
Funds for the Central Universities. SJ was also partially supported by the NSFC grants No. 12031013.  NL also  acknowledges funding from the Science and Technology Program of Shanghai, China (21JC1402900). 
CM was partially supported by China Postdoctoral Science Foundation (No. 2023M732248) and Postdoctoral Innovative Talents Support Program (No. BX20230219).

\end{document}